\documentclass[12pt]{amsart}
\usepackage{amssymb}

\usepackage[
width=31pc,
height=48pc,
margin=1in,
footskip=30pt,
]{geometry}
\usepackage{layout}

\usepackage{graphicx}
\usepackage{epsfig}
\usepackage{subcaption}
\usepackage{amsmath,amsfonts,epstopdf}
\usepackage{tikz-cd}
\usepackage{multirow}

\theoremstyle{plain}
\newtheorem{theorem}{Theorem}[section]
\newtheorem{lemma}[theorem]{Lemma}

\theoremstyle{definition}

\newtheorem{thma}{Theorem}

\numberwithin{equation}{section}

\newcommand{\R}{\mathbb{R}}
\newcommand{\N}{\mathbb{N}}
\newcommand{\Z}{\mathbb{Z}}
\newcommand{\T}{\mathbb{T}}
\newcommand{\C}{\mathbb{C}}

\renewcommand{\phi}{\varphi}

\newcommand{\hpi}{\hat{\pi}}
\newcommand{\tpi}{\tilde{\pi}}
\renewcommand{\Im}{\mathrm{Im}}
\renewcommand{\Re}{\mathrm{Re}}

\title[Existence of non-trivial embeddings of IETs into PWIs]{Existence of non-trivial embeddings of Interval Exchange Transformations into Piecewise Isometries}
\author{Pedro Peres and Ana Rodrigues}
\address{Department of Mathematics\\University of Exeter\\Exeter EX4 4QF, UK}
\begin{document}

\maketitle

\begin{abstract}
	We prove that almost every interval exchange transformation, with an associated translation surface of genus $g\geq 2$, can be non-trivially and isometrically embedded in a family of piecewise isometries. In particular this proves the existence of invariant curves for piecewise isometries, reminiscent of KAM curves for area preserving maps, which are not unions of circle arcs or line segments.
\end{abstract}

\section{Introduction}

An \textit{interval exchange transformation (IET)} is a bijective piecewise order preserving isometry $f$ of an interval $I \subset \R$. Specifically $I$ is partitioned into subintervals $\{I_{\alpha}\}_{\alpha \in \mathcal{A}}$, indexed over a finite alphabet $\mathcal{A}$ of $d\geq 2$ symbols, so that the restriction of $f$ to each subinterval is a translation.
An IET $f$ is determined by a vector $\lambda \in \R_+^{\mathcal{A}}$, with coordinates $\lambda_{\alpha}$ determining the lengths of the subintervals $I_{\alpha}$, and an irreducible permutation $\pi$ which describes the ordering of the subintervals before and after applying $f$. We write $f=f_{\lambda,\pi}$ and also denote an IET by the pair $(I,f_{\lambda,\pi})$.

IETs were studied for instance in \cite{Ke, V1, V4} and are reasonably well understood. In \cite{M,V1} Masur and Veech proved that a typical IET is uniquely ergodic while Avila and Forni \cite{AF} established that a typical IET is either weakly mixing or an irrational rotation.

A \textit{translation surface} (as defined in \cite{AF}), is a surface with a finite number of conical singularities endowed with an atlas such that coordinate changes are given by translations in $\R^2$. Given an IET it is possible to associate, via a suspension construction, a translation surface, with genus $g(\mathfrak{R})$ only depending on the combinatorial properties of the underlying IET (see \cite{V1}). Indeed these maps are deeply related to geodesic flows on flat surfaces, Teichm{\"u}ller flows in moduli spaces of Abelian differentials and polygonal billiards \cite{M}.\\

\textit{Piecewise isometries (PWIs)} are higher dimensional generalizations of one dimensional interval exchange transformations.
Let $X$ be a subset of $\C$ and $\mathcal{P}=\{X_{\alpha}\}_{\alpha \in \mathcal{A}}$ be a finite partition of $X$ into convex sets (or \textit{atoms}), that is $\bigcup_{\alpha \in \mathcal{A}} X_{\alpha} = X$ and $X_{\alpha} \cap X_{\beta} = \emptyset$ for $\alpha \neq \beta$. Given a \textit{rotation vector} $\theta \in \T^{\mathcal{A}}$ (with $\T^{\mathcal{A}}$ denoting the torus $\R^{\mathcal{A}} / 2\pi \Z^{\mathcal{A}}$) and a \textit{translation vector} $\eta \in \C^{\mathcal{A}}$, we say $(X,T)$ is a \textit{piecewise isometry} if $T$ is such that
$$  T(z) := T_{\alpha}(z) = e^{i \theta_{\alpha}} z + \eta_{\alpha}, \ \textrm{if} \ z \in X_{\alpha},   $$
so that $T$ is a piecewise isometric rotation or translation (see \cite{Goetz1}).

PWIs occur naturally in the dynamics of  Hamiltonian systems with periodic kicks \cite{L,SHM} as well as outer billiards \cite{Schwartz}. They are much less understood and appear to have more sophisticated behaviour than IETs. In general, for a given PWI it is helpful to define a partition of $X$ into a \textit{regular} and an \textit{exceptional set} \cite{AG06a}.  If we consider the set given by the union $\mathcal{E}$ of all preimages of the set of discontinuities $D$, then its closure $\overline{\mathcal{E}}$ is called the \textit{exceptional  set} for the map. Its complement is called the {\it regular set} for the map and consists of disjoint polygons or disks that are periodically coded by their itinerary through the atoms of the PWI.
For instance in \cite{AKT}, Adler, Kitchens and Tresser show, for  a particular transformation with a rational rotation vector, that  the  exceptional  set  has  zero  Lebesgue  measure.   However as highlighted in \cite{AFu} there is numerical evidence that the exceptional set may have positive Lebesgue measure for typical PWIs:  certainly this is the case for transformations which are products of minimal IETs (see \cite{Ha}).\\

It is a general belief that the phase space of typical Hamiltonian systems is divided into regions of regular and chaotic motion \cite{C83}. Area preserving maps which can be obtained as Poincar{\'e} sections of Hamiltonian systems, exhibit this property as well, with KAM curves splitting the domain into regions of chaotic and periodic dynamics (see for instance \cite{MP}). A general and rigorous treatment of this has been however missing.

PWIs, which are area preserving maps that have been studied as linear models for the standard map (see \cite{Ash97}), can exhibit a similar phenomena. Unlike IETs which are typically ergodic, there  is numerical evidence, as noted  in  \cite{AG06a}, that Lebesgue measure on the exceptional set is typically not ergodic in some families of PWIs - there can be non-smooth invariant  curves that prevent  trajectories  from  spreading  across  the  whole  of  the exceptional set. For cases where the exceptional set is a union of annuli a small perturbation in the rotational parameters causes it to decompose into invariant curves and periodic orbits, a phenomena that is reminiscent of KAM curves.
An understanding of these invariant curves would thus shed light on the ergodic properties of PWIs and would be an important first step towards the study of the dynamical behaviour shared by generic PWIs and systems which are modelled by these. A proof of their existence however remained elusive for more than a decade.

The first progress was made in \cite{AG06}, where a planar PWI, with a rational rotation vector, whose generating map is a permutation of four cones was investigated, and the existence of an uncountable number of invariant polygonal curves on which the dynamics is conjugate to a transitive interval exchange was proved. The methods used however are based on calculations in a rational cyclotomic field and do not generalize for typical choices of parameters.\\

Recently, in \cite{AGPR}, we related the existence of invariant curves to the general problem of embedding IET dynamics within PWIs, of which we gave rigorous definitions.

An injective map $\gamma:I\rightarrow X$ is a {\em piecewise continuous embedding} of $(I,f)$ into $(X,T)$ if $\gamma|_{I_{\alpha}}$ is a homeomorphism for each $\alpha \in \mathcal{A}$ such that $\gamma(I_{\alpha})\subset X_{\alpha}$ and
\begin{equation*}\label{eq0s2}
\gamma \circ f(x) = T \circ \gamma(x),
\end{equation*}
for all $x\in I$. In this case note that $\gamma(I)\subset X$ is an invariant set for $(X,T)$.

If $\gamma$ is a piecewise continuous embedding that is continuous on $I$, we say it is a {\em continuous embedding} (or \textit{embedding} when this does not cause any ambiguity).

We say $\gamma$ is a {\em differentiable embedding} if it is a piecewise continuous embedding and $\gamma|_{I_{\alpha}}$ is continuously differentiable. We characterize certain differentiable embeddings as, in some sense, trivial. Given $I'\subseteq I$ we say a map $\gamma:I' \rightarrow  \C$ is an \textit{arc map} if there exists $\xi \in \C$, $r,a>0$ and $\varphi \in [0,2\pi)$ such that for all $x \in I'$,
$$  \gamma(x) =  r e^{i(a x+\varphi)}   +\xi .$$
We say an embedding $\gamma: I \rightarrow \C$ of an IET into a PWI is an \textit{arc embedding} if there exists a finite partition of $I$ into subintervals such that the restriction of $\gamma$ to each subinterval is an arc map.
We say an embedding $\gamma$ of an IET into a PWI is a \textit{linear embedding} if $\gamma$ is a piecewise linear map.
Moreover an embedding is \textit{non-trivial} if it is not an arc embedding or a linear embedding. Figure \ref{fig:invertiblePWI} shows an illustration of a non-trivial embedding.

From the definitions it is clear that the image $\gamma(I)$ of an embedding is an invariant curve for the underlying PWI and that if the embedding is non-trivial this curve is not the union of line segments or circle arcs. 
For any IET it is straightforward to construct a PWI in which it is trivially embedded. The same is not true for non-trivial embeddings, for which results have been much scarcer. Indeed it is known (see \cite{AGPR}) that continuous embeddings of minimal 2-IETs into orientation preserving PWIs are necessarily trivial and that any $3$-PWI has at most one non-trivially continuously embedded minimal $3$-IET with the same underlying permutation. Despite supporting numerical evidence to the date of this paper there was no proof of existence of non-trivial embeddings.\\

\begin{figure}[t]
	\begin{subfigure}{.49\textwidth}
		\centering
		\includegraphics[width=1\linewidth]{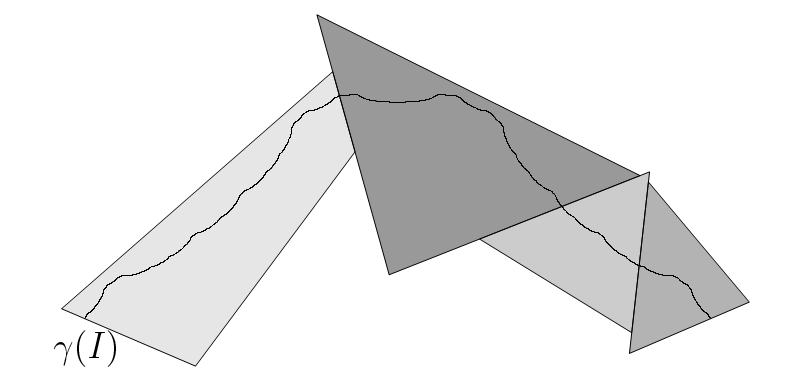}
		\caption{}
		\label{fig:invertiblePWI1}
	\end{subfigure}
	\begin{subfigure}{.49\textwidth}
		\centering
		\includegraphics[width=1\linewidth]{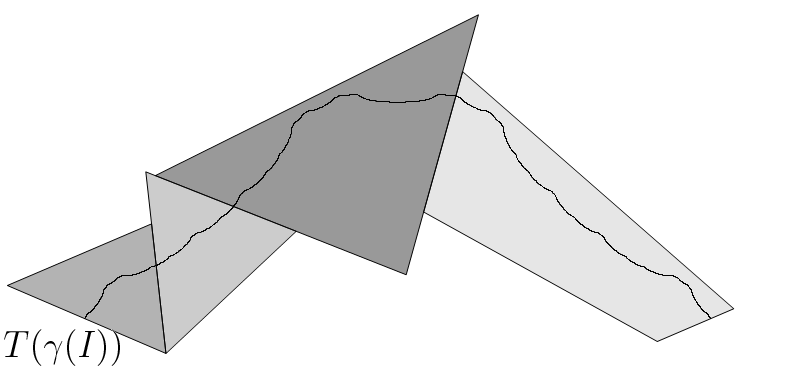}
		\caption{}
		\label{fig:invertiblePWI2}
	\end{subfigure}
	\centering
	\captionsetup{width=1\textwidth}
	\caption{An illustration of the action of a PWI $T$ with rotation vector $\theta \approx (4.85, 0.92, 1.31, 1.28)$ on its partition and on an invariant curve $\gamma(I)$. The map $\gamma$, estimated using technical tools from this paper,  is a non-trivial embedding of a self-inducing IET associated to the permutation $\tpi(j)=4-(j-1)$, $j=1,...,4$ and a translation vector of algebraic irrationals $\lambda \approx (0.43, 0.34, 0.12, 0.11)$.}
	\label{fig:invertiblePWI}
\end{figure}

In this paper we prove that a full measure set of IETs admit non-trivial embeddings into a class of PWIs thus also establishing the existence of invariant curves for PWIs which are not unions of circle arcs or line segments.

\begin{thma}\label{TA}
	For almost every IET $(I,f_{\lambda,\pi})$ satisfying $g(\mathfrak{R})\geq 2$, there exists a set $\mathcal{W} \subseteq \T^{\mathcal{A}}$, of dimension $g(\mathfrak{R})$, such that for all $\theta \in  \mathcal{W}$ there is a family $\mathcal{F}_{\theta}$, of PWIs with rotation vector $\theta$, and a map $\gamma_{\theta} : I \rightarrow \C$, which is a non-trivial and isometric embedding of $(I,f_{\lambda,\pi})$ into any $(X,T) \in \mathcal{F}_{\theta}$.
	Furthermore $\gamma_{\theta}(I)$ is an invariant curve for $(X,T)$ which is not the union of circle arcs or line segments.
\end{thma}

To prove this result we inductively define, associated to a given IET, a sequence of piecewise linear parametrized curves, which we call the \textit{breaking sequence}, dependent on a rotation vector $\theta \in \T^{\mathcal{A}}$. In particular for its  construction we define the \textit{breaking operator}, which acts on piecewise linear maps from $I$ to $\C$ by rotating particular segments of their image by a given angle. The construction also involves the \textit{Rauzy cocycle}, an important tool in the theory of IET renormalization.  
We then show that each element of the breaking sequence is a \textit{quasi-embedding} (a rigorous notion defined in Section 4) of the underlying IET into a certain sequence of piecewise isometric maps related to Rauzy induction.
Provided the breaking sequence converges to a topological embedding of the interval, this is enough to show that its limit is an embedding of the underlying IET into a family of PWIs. Hence the following step is to use tools from the theory of IET renormalization and measurable cocycles such as \textit{Zorich cocycle} \cite{Z1} and \textit{Oseledets Theorem} \cite{O} to prove this is the case for almost every $(\lambda,\pi)$ and for $\theta$ contained in a submanifold of $\T^{\mathcal{A}}$. After some further parameter exclusion to guarantee that the embedding is non-trivial we finally conclude the proof of Theorem \ref{TA}.

~

This paper is organized as follows. We start with some basic background on IETs. We then introduce the breaking sequence of curves and prove several technical lemmas which lead to the proof of Theorem \ref{fund_lemma} which states that each curve in the breaking sequence is \textit{quasi-embedded} in a certain PWI. Finally we use tools from the theory of IET renormalization to prove key results which lead to the proof of Theorem \ref{TA}.


\section{Interval Exchange Transformations}
In this section we recall some notions of the theory of interval exchange transformations following \cite{AL}, \cite{SU} and \cite{Vi06}.

We define \textit{interval exchange transformations} as in \cite{AL,Vi06}. Let $\mathcal{A}$ be an alphabet on $d \geq 2$ symbols, and let $I \subset \R$ be an interval having $0$ as left endpoint. In what follows we denote $\R^{\mathcal{A}} \simeq \R^{d}$ and $\R_+^{\mathcal{A}} \simeq \R_+^{d}$. We choose a partition $\{I_{\alpha}\}_{\alpha \in \mathcal{A}}$ of $I$ into subintervals which we assume to be closed on the left and open on the right.
An \emph{interval exchange transformation} (IET) is a bijection of $I$ defined by two data

(1) A vector $\lambda=(\lambda_{\alpha})_{\alpha \in \mathcal{A}} \in \R_+^{\mathcal{A}}$ with coordinates corresponding to the lengths of the subintervals, that is, for all $\alpha \in \mathcal{A}$, $\lambda_{\alpha} = |I_{\alpha}|$. We write $I=I(\lambda)= [0,|\lambda|)$, where $|\lambda|=\sum_{\alpha \in \mathcal{A}} \lambda_{\alpha}$.

(2) A pair $\pi=\left(\begin{array}{ll}
\pi_{0}\\
\pi_{1}
\end{array}\right)$  of bijections $\pi_{\varepsilon}:\mathcal{A}\rightarrow \{  1,...,d \}$, $\varepsilon=0,1$, describing the ordering of the subintervals $I_{\alpha}$ before and after the application of the map. This is represented as
\begin{equation*}\label{eqiet1}
\pi=\left(\begin{array}{llll}
\alpha_{1}^{0} & \alpha_{2}^{0} & ... & \alpha_{d}^{0}\\
\alpha_{1}^{1} & \alpha_{2}^{1} & ... & \alpha_{d}^{1}
\end{array}\right).
\end{equation*}
We call $\pi$ a \textit{permutation} and identify it, at times,  with its \emph{monodromy invariant} $\tilde{\pi}= \pi_1 \circ \pi_0^{-1} :\{1,...d\} \rightarrow \{1,...d\} $. We denote by $\mathfrak{S}({\mathcal{A}})$ the set of irreducible permutations, that is $\pi \in \mathfrak{S}({\mathcal{A}})$ if and only if $\tilde{\pi}(\{1,...,k\})\neq \{1,...,k\}$ for $1\leq k < d$.

Define a linear map $\Omega_{\pi}: \R^{\mathcal{A}} \rightarrow \R^{\mathcal{A}}$ by
\begin{equation}\label{eqiet3}
  \left(\Omega_{\pi}(\lambda') \right)_{\alpha \in \mathcal{A}} =  \sum_{\pi_1(\beta)<\pi_1(\alpha)}\lambda_{\beta}' - \sum_{\pi_0(\beta)<\pi_0(\alpha)}\lambda_{\beta}'   .
\end{equation}
Given a permutation $\pi \in \mathfrak{S}({\mathcal{A}})$ and $\lambda \in \R_+^{\mathcal{A}}$ the interval exchange transformation associated with this data is the map $f_{\lambda,\pi}$ that rearranges $I_\alpha$ according to $\pi$, that is
\begin{equation}\label{eqiet2}
f_{\lambda,\pi}(x)=x+ \upsilon_{\alpha},
\end{equation}
for any $x \in I_\alpha$, where $\upsilon_{\alpha}=(\Omega_{\pi}(\lambda))_{\alpha}$.

We will assume throughout the rest of this paper that $(\lambda,\pi)$ satisfies the \emph{infinite distinct orbit condition (IDOC)}, first introduced by Keane in \cite{Ke}. $(\lambda,\pi)$ satisfies the IDOC if the orbits of the endpoints of the subintervals $\{I_{\alpha}\}_{\alpha \in \mathcal{A}}$ are as disjoint as possible
\begin{equation*}\label{eqiet10}
f_{\lambda,\pi}^n\left( \sum_{\pi_0(\eta)<\pi_0(\alpha)} \lambda_{\eta}    \right) \neq \sum_{\pi_0(\eta)<\pi_0(\beta)} \lambda_{\eta},
\end{equation*}
for all $n \geq 1$ and $\alpha, \beta \in \mathcal{A}$ with $\pi_0(\beta) \neq 1$. In particular the IDOC implies minimality of $f_{\lambda,\pi}$, that is, every orbit is dense in the interval. 

\subsection{Rauzy induction} We define \emph{Rauzy induction} as in \cite{Vi06}. Let $(\lambda,\pi) \in \R_+^{\mathcal{A}} \times \mathfrak{S}({\mathcal{A}})$. For $\varepsilon=0,1$, denote by $\beta_{\varepsilon}$ the last symbol in the expression of $\pi_{\varepsilon}$, that is
\begin{equation*}\label{eqiet4}
\beta_{\varepsilon}=\pi_{\varepsilon}^{-1}(d)=\alpha_{d}^{\varepsilon}.
\end{equation*}
Assume the intervals $I_{\beta_{0}}$ and $I_{\beta_{1}}$ have different lengths. We say that $(\lambda,\pi)$ is of \emph{type} $0$ if $\lambda_{\beta_{0}}>\lambda_{\beta_{1}}$ and is of \emph{type} $1$ if $\lambda_{\beta_{0}}<\lambda_{\beta_{1}}$. The largest interval is called \emph{winner} and the smallest \emph{loser} of $(\lambda,\pi)$. Let $I^{(1)}$ be interval obtained by removing the loser from $I(\lambda)$:
$$I^{(1)}= \left[0,|\lambda|-\min(|\lambda_{\beta_{0}}|,|\lambda_{\beta_{1}}|)  \} \right).$$
 The first return map of $f_{\lambda,\pi}$ to the subinterval $I^{(1)}$ is again an IET, $f_{\lambda^{(1)},\pi^{(1)}}$, where the parameters $(\lambda^{(1)},\pi^{(1)})$ are defined as follows. If $(\lambda,\pi)$ is of type $0$ then
 \begin{equation}\label{eqiet6}
 \pi^{(1)}=\left(\begin{array}{ll}
 \pi_{0}^{(1)}\\
 \pi_{1}^{(1)}
 \end{array}\right)=\left(\begin{array}{llllllll}
 \alpha_{1}^{0} & ... & \alpha_{k-1}^{0} & \alpha_{k}^{0} & \alpha_{k+1}^{0} & ... & ... & \beta_{0}\\
 \alpha_{1}^{1} & ... & \alpha_{k-1}^{1} & \beta_{0} & \beta_{1} & \alpha_{k+1}^{1} & ... & \alpha_{d-1}^{1}
 \end{array}\right).
 \end{equation}
 where $k \in \{1,...,d-1\}$ is defined by $\alpha_{k}^{1}=\beta_{0}$, and $\lambda^{(1)}=(\lambda_{\alpha}^{(1)})_{\alpha \in \mathcal{A}}$, where
  \begin{equation*}\label{eqiet7}
 \lambda_{\alpha}^{(1)}=\lambda_{\alpha} \ \textrm{for} \ \alpha \neq \beta_{0}, \ \textrm{and} \  \lambda_{\beta_{0}}^{(1)}=\lambda_{\beta_{0}} - \lambda_{\beta_{1}}.
  \end{equation*}
  If $(\lambda,\pi)$ is of type $1$ then
  \begin{equation}\label{eqiet8}
  \pi^{(1)}=\left(\begin{array}{ll}
  \pi_{0}^{(1)}\\
  \pi_{1}^{(1)}
  \end{array}\right)=\left(\begin{array}{llllllll}
  \alpha_{1}^{0} & ... & \alpha_{k-1}^{0} & \beta_{1} & \beta_{0} & \alpha_{k+1}^{0} & ... & \alpha_{d-1}^{0}\\
  \alpha_{1}^{1} & ... & \alpha_{k-1}^{1} & \alpha_{k}^{1} & \alpha_{k+1}^{1} & ... & ... & \beta_{1}\\
  \end{array}\right).
  \end{equation}
  where $k \in \{1,...,d-1\}$ is defined by $\alpha_{k}^{0}=\beta_{1}$, and $\lambda^{(1)}=(\lambda_{\alpha}^{(1)})_{\alpha \in \mathcal{A}}$, where
  \begin{equation}\label{eqiet9}
  \lambda_{\alpha}^{(1)}=\lambda_{\alpha} \ \textrm{for} \ \alpha \neq \beta_{1}, \ \textrm{and} \  \lambda_{\beta_{1}}^{(1)}=\lambda_{\beta_{1}} - \lambda_{\beta_{0}}.
  \end{equation}
  The map  $\mathcal{R}:\R_+^{\mathcal{A}} \times \mathfrak{S}({\mathcal{A}}) \rightarrow \R_+^{\mathcal{A}} \times \mathfrak{S}({\mathcal{A}})$ defined by $\mathcal{R}(\lambda,\pi)= (\lambda^{(1)},\pi^{(1)})$ is called the \textit{Rauzy induction map}.
  
  The IDOC condition assures that the iterates $\mathcal{R}^n$ are defined for all $n\geq 0$. We denote 
    \begin{equation}\label{Rauzindn}
    \mathcal{R}^n(\lambda,\pi)=(\lambda^{(n)},\pi^{(n)}),
    \end{equation}
    with  $\pi^{(n)}= (\begin{array}{ll} \pi_{0}^{(n)}  &  \pi_{1}^{(n)} \end{array})^T$. Furthermore we denote by $\beta_{\varepsilon,n}$ the last symbol in the expression of $\pi_{\varepsilon}^{(n)}$, $\varepsilon(n)$ the type of  $f_{\lambda^{(n)},\pi^{(n)}}$,  by $I^{(n)}$ its domain  and by $\{I_{\alpha}^{(n)}\}_{\alpha \in \mathcal{A}}$ its partition in subintervals, for $n\geq 0$. We also denote the translation vector of $f_{\lambda^{(n)},\pi^{(n)}}$ by  $\upsilon^{( n  )} = \Omega_{\pi^{( n  )}}(\lambda^{(n )})$.
  
  \subsection{Rauzy classes}
  The \emph{Rauzy class} (see \cite{Vi06}) of a permutation $\pi\in \mathfrak{S}({\mathcal{A}})$, is the set $\mathfrak{R}(\pi)$ of all $\pi^{(1)}\in \mathfrak{S}({\mathcal{A}})$ such that there exist $\lambda,\lambda^{(1)} \in\R_+^{\mathcal{A}}$ and $n\in \N$ such that $\mathcal{R}^n(\lambda,\pi)=(\lambda^{(1)},\pi^{(1)})$. A Rauzy class $\mathfrak{R}$  can be visualized in terms of a directed labeled graph, the \emph{Rauzy graph} (see \cite{SU}). Its vertices are in bijection with $\mathfrak{R}$ and it is formed by arrows that connect permutations which are obtained one from another by \eqref{eqiet6} and \eqref{eqiet8} labeled respectively by $0$ or $1$ according to the type of the induction.
  A \emph{path} $\gamma=(\gamma_1,...,\gamma_n)$ is a sequence of compatible arrows of the Rauzy graph, that is, such that the starting vertex of $\gamma_{i+1}$ is the ending vertex of $\gamma_{i}$, $i=1,...,n-1$. We say a path is \emph{closed} if the starting vertex of $\gamma_1$ is the ending vertex of $\gamma_n$. The set of all paths in this graph is denoted by $\Pi(\mathfrak{R})$.
  
\subsection{Rauzy Cocycle} We define \textit{Rauzy cocycle} as in \cite{AL}. A linear cocycle is a pair $(T,A)$, where $T:\mathfrak{X} \rightarrow \mathfrak{X}$ and $A:\mathfrak{X} \rightarrow GL(p,\R)$, viewed as a linear skew-product $(x,v)\rightarrow (T(x),A(x)\cdot v)$ on $\mathfrak{X} \times \R^p$, $p\in \N$. Notice that $(T,A)^n=(T^n,A^{(n)})$, where
\begin{equation*}\label{eqiet11}
A^{(n)}(x)=A(T^{n-1}(x))\cdot ... \cdot A(x), \quad n \geq 0.
\end{equation*}
In what follows we denote $SL({\mathcal{A}},\Z) \simeq SL(d,\Z)$. Let $\| \cdot \|$ denote a matrix norm on $SL({\mathcal{A}},\Z)$. Let $\log^+y= \max\{\log(y) ,0  \}$ for any $y>0$. If $(\mathfrak{X}, \mu)$ is a probability space, $\mu$ is an ergodic measure for $T$ and
$$  \int_{\mathfrak{X}} \log^+\| A(x) \| d \mu(x) < + \infty,$$
we say $(T,A)$ is a measurable cocycle.\\

Denote by $E_{\alpha \beta}$ the elementary matrix $(\delta_{i\alpha}\delta_{j\beta})_{i\geq 1, j\leq d}$. Let $\mathfrak{R}\subseteq \mathfrak{S}({\mathcal{A}})$ be a Rauzy class. To any given path $\gamma \in \Pi(\mathfrak{R})$ we associate a matrix $B_{\gamma} \in SL({\mathcal{A}},\Z)$, defined inductively as follows\vspace*{0.2cm}\\
i) If $\gamma$ is an arrow labeled by $0$, set $B_{\gamma}=\mathbf{1}_d+E_{\alpha(1) \alpha(0)}$, with $\mathbf{1}_d$ denoting the $d\times d$ identity matrix;\vspace*{0.2cm}\\
ii) If $\gamma$ is an arrow labeled by $1$, set $B_{\gamma}=\mathbf{1}_d+E_{\alpha(0) \alpha(1)}$;\vspace*{0.2cm}\\
iii) If  $\gamma=(\gamma_1,...,\gamma_n)$, set $B_{\gamma}=B_{\gamma_n}...B_{\gamma_1}$.\\

We denote by $\gamma(\lambda,\pi) \in \Pi(\mathfrak{R}(\pi))$, the arrow in the Rauzy graph starting at $\pi$ labeled by the type of $(\lambda,\pi)$.

Define the function $B_R: \R_+^{\mathcal{A}} \times \mathfrak{R}  \rightarrow SL({\mathcal{A}},\Z)$ such that $B_R(\lambda,\pi)=B_{\gamma(\lambda,\pi)}$. The \emph{Rauzy cocycle} is the linear cocycle over the Rauzy induction $(\mathcal{R},B_R)$ on $\R_+^{\mathcal{A}} \times \mathfrak{R} \times \R^{\mathcal{A}}$. Note that $(\mathcal{R},B_R)^n=(\mathcal{R}^n,B_R^{(n)})$, where
\begin{equation}\label{BR}
B_R^{(0)}(\lambda,\pi)=\mathbf{1}_d, \quad B_R^{(n)}(\lambda,\pi)= B(\lambda^{(n-1)},\pi^{(n-1)})\cdot...\cdot B(\lambda^{(1)},\pi^{(1)})\cdot B(\lambda,\pi),
\end{equation}
for $n\geq 1$, with $(\lambda^{(n)},\pi^{(n)})$ as in \eqref{Rauzindn}. Note that, we have
\begin{equation*}\label{lambdaBR}
\lambda= \left(B_R^{(n)}(\lambda,\pi) \right)^*\cdot \lambda^{(n)},
\end{equation*}
for all $n\geq0$, where $^*$ denotes the transpose operator.

We now stress an important property of the Rauzy cocycle (see \cite{Vi06}). For any $n\geq0$ and $x \in I^{(n)}$, let $r_{\lambda,\pi}^n(x)$ denote the first return time of $x$ by $f_{\lambda,\pi}$ to $I^{(n)}$. Note that $r_{\lambda,\pi}^n(x)$ is constant on each $I^{(n)}_{\alpha}$, for any $\alpha \in \mathcal{A}$. We write $r_{\lambda,\pi}^n(I^{(n)}_{\alpha})$  to mean $r_{\lambda,\pi}^n(x)$, for any $x \in I^{(n)}_{\alpha}$.

Each entry $\left(B_R^{(n)}(\lambda,\pi)\right)_{\alpha,\beta}$ of the matrix $B_R^{(n)}(\lambda,\pi)$ counts the number of visits of $I_{\alpha}^{(n)}$ to the interval $I_\beta$ up to the $r_{\lambda,\pi}^n(I^{(n)}_{\alpha})$-th iterate of $f_{\lambda,\pi}$.  That is for every $\alpha, \beta \in \mathcal{A}$ and every $n \geq 1$,
\begin{equation*}\label{timeBR}
\left(B_R^{(n)}(\lambda,\pi)\right)_{\alpha,\beta}= \textrm{card} \left\{0\leq j < r^{n}_{\lambda,\pi} (I_{\alpha}^{(n)}):f_{\lambda,\pi}(I_{\alpha}^{(n)})  \subset I_\beta   \right\}.
\end{equation*}

\subsection{Projection of the Rauzy cocycle on $\mathbb{T}^\mathcal{A}$}
Denote $\Z^{\mathcal{A}} \simeq \Z^{d}$ and let $\mathbb{T}^{\mathcal{A}} \simeq \mathbb{T}^{d}$ be the $d$-dimensional torus $\R^{\mathcal{A}} / 2\pi \Z^{\mathcal{A}}$.  Furthermore, let $p:\R^{\mathcal{A}} \rightarrow  \mathbb{T}^{\mathcal{A}}$ be the natural projection,
$$ p(v) = \left(  \left( v    \right)_{\alpha}\mod 2 \pi    \right)_{\alpha \in \mathcal{A}}, \quad \textrm{ for all} \ v \in \R^{\mathcal{A}}.$$
We sometimes denote $p(v)= v \mod 2 \pi$.

We introduce the \emph{projection of the Rauzy cocycle on $\mathbb{T}^{\mathcal{A}}$} as the application $B_{\mathbb{T}^{\mathcal{A}}} : \R_+^{\mathcal{A}} \times \mathfrak{R} \times \mathbb{T}^{\mathcal{A}} \rightarrow \mathbb{T}^{\mathcal{A}}$ such that $B_{\mathbb{T}^{\mathcal{A}}}(\lambda,\pi)\cdot \theta = p(B_{R}(\lambda,\pi) \cdot v)$ , for any $(\lambda,\pi) \in \R_+^{\mathcal{A}} \times \mathfrak{R}$, $n \geq 0$ and $\theta \in \mathbb{T}^{\mathcal{A}}$, with $v \in p^{-1}(\theta)$. Note that, as $B_{R}$ is an integral cocycle, for any $v, v' \in p^{-1}(\theta)$ we have $p(  B_{R}(\lambda,\pi) \cdot v) = p(  B_{R}(\lambda,\pi) \cdot v')$ and thus the map $B_{\mathbb{T}^{\mathcal{A}}}$ is well defined. We also denote
\begin{equation}\label{crauzytor}
B_{\mathbb{T}^{\mathcal{A}}}^{(n)}(\lambda,\pi) \cdot \theta = B_{R}^{(n)}(\lambda,\pi) \cdot v \mod 2\pi,
\end{equation}
for any $n \geq 0$ and $\theta \in \mathbb{T}^{\mathcal{A}}$, with $v \in p^{-1}(\theta)$.

\section{Breaking sequence}

In this section we define the \emph{breaking sequence}, a sequence of curves associated to IET parameters $(\lambda,\pi) \in \R_+^{\mathcal{A}} \times \mathfrak{S}({\mathcal{A}})$ and a rotational parameter $\theta \in \T^{\mathcal{A}}$ via the \emph{breaking operator}, an operator acting on the space of piecewise linear curves. We then relate the dynamics of a breaking sequence and that of the underlying IET.


Given $\ell>0$ we denote the class of continuous piecewise linear maps $\gamma:[0,\ell)\rightarrow \C$ such that all $x$ in the domain of differentiability of $\gamma$, satisfy $|\gamma(x)'|=1$ by $\mathcal{PL}(\ell)$.
Note that for any $\gamma \in \mathcal{PL}(\ell)$, its image $\gamma([0,\ell))$ has an arc length equal to $\ell$.




We say that a sequence of mutually disjoint intervals $\{J_n\}$ is an \emph{ordered sequence of disjoint intervals} if whenever $m<m'$, we have $x<x'$ for all $x \in J_{m}$ and $x' \in J_{m'}$.

We define an operator that acts on $\mathcal{PL}(\ell)$, which given a sequence of subintervals of its domain, takes a curve, and rotates by a fixed angle the pieces corresponding to these subintervals. Visually and informally the action of this operator is to rigidly "break" the curve at these segments.

Consider a map $\gamma \in \mathcal{PL}(\ell)$ a real number $\varphi \in [-\pi,\pi)$, and an ordered sequence of disjoint intervals $J=\{J_k\}_{0\leq k \leq r-1}$ of equal length $\Delta\in (0,\ell / r)$. We write $J_k=[y_k,y_k+\Delta)\subset \R$, where $y_k+\Delta\leq y_{k+1}$ and $k \in \{0,...,r-1\}$. We define the \textit{breaking operator} $\mathfrak{Br}(\varphi,J):\mathcal{PL}(\ell) \rightarrow \mathcal{B}([0,\ell),\C)$ as
\begin{equation}\label{breq1}
\mathfrak{Br}(\varphi,J)\cdot \gamma(x)=
\left\{
\begin{array}{ll}\vspace{0.2cm}
\gamma(x), & x \in [0,y_0),\\\vspace{0.2cm}
\gamma(x)\cdot e^{i\varphi}+\overline{\epsilon}_k(\varphi,J), & x \in [y_k,y_k+\Delta),\\
\gamma(x)+\underline{\epsilon}_k(\varphi,J), & x \in [y_{k}+\Delta,y_{k+1}),\\
\end{array}
\right.
\end{equation}
for $k \in \{0,...,r-1\}$, where $y_r=\ell$,
\begin{equation}\label{breq3}
\overline{\epsilon}_0=\gamma(x_0)(1-e^{i \varphi}), \quad \overline{\epsilon_k}=\gamma(y_k)(1-e^{i \varphi})+\underline{\epsilon}_{k-1},
\end{equation}
and also
\begin{equation}\label{breq2}
\underline{\epsilon}_0=(\gamma(y_0)-\gamma(y_0+\Delta))(1-e^{i \varphi}), \quad \underline{\epsilon}_k=\overline{\epsilon}_k-\gamma(y_k+\Delta)(1-e^{i \varphi}).
\end{equation}

We first need to  show that for all $\ell >0$ and $\varphi \in [-\pi,\pi)$, $\mathfrak{Br}(\varphi,J)$ maps $\mathcal{PL}(\ell)$ into a subset of $\mathcal{PL}(\ell)$. We do this in our next lemma.

\begin{lemma}\label{lPL}
	If $\ell>0$, $\gamma \in \mathcal{PL}(\ell)$, $\varphi \in [-\pi,\pi)$ and $J$ is an ordered sequence of disjoint subintervals of $[0, \ell)$ with length $\Delta>0$, then $\mathfrak{Br}(\varphi,J)(\mathcal{PL}(\ell)) \subseteq \mathcal{PL}(\ell)$.
\end{lemma}

\begin{proof}

Let $\gamma \in \mathcal{PL}(\ell)$. It is clear that $\mathfrak{Br}(\varphi,J)\cdot \gamma$ is  piecewise linear  and  continuous. In particular, it is semi-differentiable. Denote by $\partial_{-}$ and $\partial_{+}$ its left and right derivative, respectively.

 Given $x \in (0,\ell)$ we have
 $$|\partial_{-} \left( \mathfrak{Br}(\varphi,J)\cdot \gamma\right)(x)|=|\partial_{+} \left( \mathfrak{Br}(\varphi,J)\cdot \gamma\right)(x)|=|\partial_{+}\gamma(x)|.$$
 Since $\gamma \in \mathcal{PL}(\ell)$, $|\partial_{+}\gamma'(x)|=1$ and hence, if $\mathfrak{Br}(\varphi,J)\cdot \gamma$ is differentiable at $x$ we must have $|\left( \mathfrak{Br}(\varphi,J)\cdot \gamma\right)'(x)|=1$. This finishes our proof.

\end{proof}

We will later need the estimate in the next lemma.

\begin{lemma}\label{brth1}
	Let $\ell>0$, $\gamma \in \mathcal{PL}(\ell)$, $\varphi \in [-\pi,\pi)$, $\Delta<\ell$ be a positive constant and $J$ be an ordered sequence of disjoint intervals of length $\Delta$. For all $k \in \N$ we have
	\begin{equation*}\label{breq4a}
	\max\left(|\overline{\epsilon}_k|,|\underline{\epsilon}_k|\right) \leq  2\ell \sin\left|\frac{\varphi}{2}\right|.
	\end{equation*}	
\end{lemma}

\begin{proof}
	Let $r$ be the number of subintervals in $J$ and denote $J=\{[y_k,y_k+\Delta)\}_{0\leq k<r}$. 	
	
	By inserting \eqref{breq2} in \eqref{breq3} it is clear that, for any $1\leq k<r$, we have
	$$  \overline{\epsilon}_{k}= (\gamma(y_k)  - \gamma(y_{k-1}+\Delta) )(1-e^{i\varphi}) + \overline{\epsilon}_{k-1}, $$
	and applying the triangle inequality we get, for any $1\leq k<r$, that
	$$ |\overline{\epsilon}_{k}|\leq |1-e^{i\varphi}| \left[  |\gamma(y_k)| + \sum_{l=0}^{k-1}|\gamma(y_l)-\gamma(y_{l}+\Delta)|    \right]. $$

	As $|1-e^{i\varphi}| = |\sin(\varphi/2)|$, $y_k\leq \ell-(r-k)\Delta$ and
	$ |\gamma(y_l)-\gamma(y_{l}+\Delta)| \leq \gamma([y_k,y_k+\Delta))\leq \Delta  $ we get
	as $r \Delta\leq \ell$ 
	$$ |\overline{\epsilon}_k| \leq  2\ell \sin\left|\frac{\varphi}{2}\right|   . $$

	It is also clear from \eqref{breq3} and \eqref{breq2}   applying the triangle inequality 
	that for any $1\leq k<r$ we have 
	$$ |\underline{\epsilon}_{k}|\leq |1-e^{i\varphi}| \sum_{l=0}^{k}|\gamma(y_l)-\gamma(y_{l}+\Delta)|, $$
	and in a similar way as before we can prove that $ |\underline{\epsilon}_{k}|\leq k \Delta\sin|\varphi/2| \leq \ell \sin\left|\frac{\varphi}{2}\right|.$ 
	
\end{proof}

Let $\mathcal{J}$ be a collection of mutually disjoint intervals. We say an ordered sequence of intervals $\{J_n\}$ is an \emph{ordering} of $\mathcal{J}$ if for all $J \in \mathcal{J}$ there is a unique $m$ such that $J_m=J$. Note that if $\mathcal{J}$ is a finite collection then it has a unique ordering.


Given $(\lambda,\pi) \in \R_+^{\mathcal{A}} \times \mathfrak{S}({\mathcal{A}})$, consider the collection of sets
\begin{equation}\label{Jn}
\mathcal{J}(n)=\left\{   f_{\lambda,\pi}^k\left(I^{(n-1)}\backslash I^{(n)}\right)   \right\}_{0\leq k<r(n-1)},
\end{equation}
where $r(n-1) = r_{\lambda,\pi}^{n-1}\left(I_{\beta_{0,n-1}}^{(n-1)}\right)$ and 
$$ \beta_{\varepsilon,m}=\left(   \pi_{\varepsilon}^{(m)}  \right)^{-1}(d).$$ 
It is clear that for all $n\geq 1$,  $r(n-1)$ is equal to the smallest $r\geq 1$ such that $f_{\lambda,\pi}^k(I^{(n-1)}\backslash I^{(n)})\subset I^{(n)}$.
We denote by $J^{(n)}=\{  J_k^{(n)}   \}_{0\leq k <r(n-1)}$ be the ordering of $\mathcal{J}(n)$, for all $n \geq 1$.

Recall the projection of the Rauzy cocycle in $\T^{\mathcal{A}}$. Given $\theta \in \T^{\mathcal{A}}$ let
\begin{equation}\label{thetan}
\theta^{(0)}=\theta, \quad \theta^{(n)}=B_{\T^{\mathcal{A}}}^{(n)}(\lambda,\pi)\cdot \theta,
\end{equation}

We define the \textit{breaking sequence} as a sequence of piecewise linear curves $\{\gamma_{\theta}^{(n)}(x)\} \in \mathcal{PL}(\ell)$, such that
\begin{equation}\label{gamman}
\begin{array}{l}\vspace*{0.2cm}
\gamma_{\theta}^{(0)}(x)= x ,\\
\gamma_{\theta}^{(n)}(x)= \mathfrak{Br}\left( \theta_{\beta_{1,n-1}}^{(n-1)},  J^{(n)}\right) \cdot \gamma_{\theta}^{(n-1)}(x),
\end{array}
\end{equation}
for all $x \in [0,|\lambda|)$ and $n\geq1$.

Each map is a parametrization of a curve and is obtained by successively applying the breaking operator with angles $\theta_{\beta_{1,n-1}}^{(n-1)}$ and segments $J^{(n)}$. Note that the number of these segments will increase while their lengths will decrease as $n\rightarrow +\infty$. In this way this sequence of curves is related both to the IET $f_{\lambda,\pi}$ and to a PWI with rotation vector $\theta$.

Denote by $\Theta_{\lambda,\pi}$ the set of all $\theta \in \T^{\mathcal{A}}$ such that for all $n \geq 0$, $\gamma_{\theta}^{(n)}:I \rightarrow \C$ is an injective map. Throughout the rest of his paper we will consider $\gamma^{(n)} = \gamma_{\theta}^{(n)}$ with $\theta \in \Theta_{\lambda,\pi}$.


~

The monodromy invariant of the permutation $\pi^{(m)}$ is the bijection ${\tilde \pi}^{(m)}:   \{ 1,\cdots, d \} \rightarrow \{ 1,\cdots, d \}$ such that
${\tilde \pi}^{(m)}= \pi_1^{(m)} \circ (\pi_0^{(m)})^{-1}$. We denote its inverse by
${\hat \pi}^{(m)}= \pi_0^{(m)} \circ (\pi_1^{(m)})^{-1}.$
We write
\begin{equation}\label{x0x1}
\begin{array}{ll}\vspace*{0.2cm}
x_{\varepsilon,j}^{(m)}= \sum_{\pi_{\varepsilon}^{(m)} (\alpha) \leq j} \lambda_{\alpha}^{(m)},
\end{array}
\end{equation}
for $\varepsilon=0,1$, where $x_{0,j}^{(m)}$ denotes the $j$-th endpoint of the partition associated to to $f_{\lambda^{(m)}, \pi^{(m)} }$, this is $\{I_{\alpha}^{(m)}\}_{\alpha \in \mathcal{A}}$, and $x_{1,j}^{(m)}$ denotes the $j$-th endpoint of the image of this partition by  $f_{\lambda^{(m)}, \pi^{(m)} }$. Furthermore we denote their image by $\gamma^{(n)}$  as $\gamma_{\varepsilon, j}^{n,m} = \gamma^{(n)} \left( x_{\varepsilon,j}^{(m)}  \right).$


We may now define points $\xi_j^{n,m} \in \C$ recursively as follows
	\begin{equation}\label{xis}
	\begin{array}{ll}\vspace*{0.2cm}
	\xi_d^{n,m} = \gamma_{0,d}^{n,m}, \\
	\xi_j^{n,m}= e^{i \theta^{(m)}_{ (\pi_1^{(m)})^{-1} (j+1)  }}   \left( \gamma^{n,m}_{0,  {\hat \pi}^{(m)}(j+1) -1  }   -  \gamma^{n,m}_{0,  {\hat \pi}^{(m)}(j+1) }   \right)+\xi_{j+1}^{n,m}.
	\end{array}
	\end{equation}
	For all $\alpha \in {\mathcal A}$, $n \in {\mathbb N}$, $0\leq m \leq n$ and $z \in \C$,  we define a map,
	\begin{equation}\label{hatTnm}
		{\hat T}_{\alpha}^{(n,m)} (z)=  e^{ i \theta_{\alpha}^{(m)} }   \left( z - \gamma_{0, \pi_0^{(m)} (\alpha)  }^{n,m}   \right)   +   \xi_{ \pi_1^{(m)} (\alpha) }^{n,m}.
	\end{equation}
	
The isometries $\hat{T}_{\alpha}^{(n,m)}$ act on the segments $\gamma^{(n)}(I_{\alpha}^{(m)})$ by rearranging their order according to the permutation $\pi^{(m)}$, via rotations by angles $\theta_{\alpha}^{(m)}$.
The right endpoint $\gamma_{0,\hpi^{(m)}(d)}^{n,m}$ of the segment $\gamma^{(n)}(I_{\beta_{1,m}}^{(m)})$ is mapped to the right endpoint $\xi_{d}^{n,m}$ of $\gamma^{(n)}(I^{(m)})$. For $j<d$, the right endpoint $\gamma_{0,\hpi^{(m)}(j)}^{n,m}$ of $\gamma^{(n)}(I_{\hpi^{(m)}(j)}^{(m)})$ is mapped to the left endpoint $\xi_{j}^{n,m}$ of the image by $\hat{T}_{\hpi^{(m)}(j+1)}^{(n,m)}$ of $\gamma^{(n)}(I_{\hpi^{(m)}(j+1)}^{(m)})$.
In this way, the union over $\alpha \in \mathcal{A}$, of all $\hat{T}_{\alpha}^{(n,m)}(\gamma^{(n)}(I_{\alpha}^{(m)}))$ is a continuous curve which \textit{a priori} may not coincide with $\gamma^{(n)}(I^{(m)})$.

	We also define inductively a map $T_{\alpha}^{(n,m)}$ as follows:
	\begin{equation}\label{Tnn}
	T_{\alpha}^{(n,n)} (z)=  {\hat T}_{\alpha}^{(n,n)} (z).
	\end{equation}
	For $z \in \C$, $0 < m \leq n$, if $\varepsilon (m-1) = 0$ then
	\begin{equation}\label{Tnm0}
	T_{\alpha}^{(n,m-1)} (z)= \begin{cases}\vspace*{0.2cm}
	T_{\alpha}^{(n,m)} (z), \quad \alpha \neq \beta_{1,m-1} ,\\
	\left( T_{  \beta_{0,m-1}   }^{(n,m)} \right)^{-1} \circ T_{\alpha}^{(n,m)} (z), \quad \alpha = \beta_{1,m-1},
	\end{cases}
	\end{equation}
	if $\varepsilon (m-1) = 1$ then
	\begin{equation}\label{Tnm1}
	T_{\alpha}^{(n,m-1)} (z)= \begin{cases}\vspace*{0.2cm}
	T_{\alpha}^{(n,m)} (z)  ,\quad \alpha \neq \beta_{0,m-1}, \\
	T_{\alpha}^{(n,m)}   \circ  \left( T_{ \beta_{1,m-1}     }^{(n,m)} \right)^{-1}(z) ,\quad \alpha = \beta_{0,m-1}.
	\end{cases}
	\end{equation}
	
	Finally, we define a map $T^{(n,m)} : \gamma^{(n)} \left( I^{(m)}  \right) \rightarrow {\mathbb C}$ by
	\begin{equation*}\label{themap}
	T^{(n,m)} (z) = T_{\alpha}^{(n,m)} (z), \quad z \in \gamma^{(n)} \left( I_{\alpha}^{(m)} \right).
	\end{equation*}
	
	To understand the rationale behind the inductive procedure used to define $T^{(n,m)}$, consider first the map $f_{\lambda^{(m)},\pi^{(m)},\alpha}:\R\rightarrow \R$ such that $f_{\lambda^{(m)},\pi^{(m)},\alpha}(x)= x+\upsilon_{\alpha}^{(m)}$.
	If $\theta=0$, by the definition of breaking sequence note that$ \gamma_{\theta}^{(n)}(x)=x$, for all $x \in I$ and $n\geq 0$.
	Consequently, we have $\gamma_{\varepsilon, j }^{n,m}=x_{\varepsilon, j }^{(m)}$, $\xi_{j }^{n,m}=x_{1, j }^{(m)}$ and thus, for all $z \in\C$, we have
	$$  \hat{T}_{\alpha}^{(n,m)}(z) =  f_{\lambda^{(m)},\pi^{(m)},\alpha}(\Re(z)) + i \Im(z)  .$$
	For $0<m\leq n$ and $\varepsilon(m-1)=0$, \eqref{Tnm0} gives
	\begin{equation*}\label{}
	f_{\lambda^{(m-1)},\pi^{(m-1)},\alpha} (\Re(z))= \begin{cases}\vspace*{0.2cm}
	f_{\lambda^{(m)},\pi^{(m)},\alpha} (\Re(z)), \quad \alpha \neq \beta_{1,m-1} ,\\
	f_{\lambda^{(m)},\pi^{(m)},\beta_{0,m-1}}^{-1} \circ f_{\lambda^{(m)},\pi^{(m)},\alpha} (\Re(z)), \quad \alpha = \beta_{1,m-1},
	\end{cases}
	\end{equation*}
	and as $f_{\lambda^{(m)},\pi^{(m)}}(x)= f_{\lambda^{(m)},\pi^{(m)},\alpha}(x)$, when $x \in I_{\alpha}^{(m)}$, these identities can be easily verified to be equivalent to Rauzy induction in this case.
	An analogous set of identities can also be obtained for the case $\varepsilon(m-1)=1$. Also note that for this example we have $\hat{T}_{\alpha}^{(n,m)}=T_{\alpha}^{(n,m)}$. This is no coincidence and indeed later we will prove that this identity holds in general.
	In this way \eqref{Tnm0} and \eqref{Tnm1} are a generalization of Rauzy induction and hence $\{T^{(n,m)} \}_{n \geq 0}$ is a sequence of maps defined on $\gamma^{(n)}(I^{(m)})$ which preserves this inductive structure.
	
	For the remainder of this section we state and prove several lemmas which serve as technical tools for our next section where we explore the relation between $\hat{T}_{\alpha}^{(n,m)}$, $T_{\alpha}^{(n,m)}$, $\gamma^{(n)}$ and $f_{\lambda^{(m)},\pi^{(m)}}$.
	The following lemma gives useful expressions for compositions of $\hat{T}_{\alpha}^{(n,m)}$ which are related to inductive procedure used to define $T_{\alpha}^{(n,m)}$.


	\begin{lemma}\label{LA}
	For all $n\geq 1$, $0< m \leq n$ and $z \in \C$ if $\varepsilon(m-1)=0$ then	
	\begin{equation*}\label{A1'}
	\left( {\hat T}_{ \beta_{0,m-1}  }^{(n,m)} \right)^{-1} \circ {\hat T}_{\beta_{1,m-1}}^{(n,m)} (z) = e^{i\theta_{\beta_{1,m-1}}^{(m-1)}}\left( z - \gamma_{0,{\hat  \pi}^{(m-1)}(d)-1}^{n,m-1}    \right)   +  \gamma_{1,d-1}^{n,m-1}  .
	\end{equation*}
	 and if  $\varepsilon(m-1)=1$ then	
	\begin{equation*}\label{A2'}
	 {\hat T}_{ \beta_{0,m-1}  }^{(n,m)}  \circ \left( {\hat T}_{\beta_{1,m-1}}^{(n,m)} \right)^{-1}(z) = e^{i\theta_{\beta_{0,m-1}}^{(m-1)}}\left( z -  \gamma_{0,d-1}^{n,m-1}   \right)  + \xi_{{\tilde  \pi}^{(m-1)}(d)-1}^{n,m} .
	\end{equation*}
	\end{lemma}

	\begin{proof}
	Assume first  $\varepsilon(m-1)=0$. It is clear that 
$\pi_{0}^{(m-1)}= \pi_{0}^{(m)}$, $\pi_{1}^{(m)}(\beta_{1,m-1})= \pi_{1}^{(m)}(\beta_{0,m})+1$ and we get	
	\begin{equation}\label{A3}
	\xi_{\pi_{1}^{(m)}(\beta_{1,m-1})}^{n,m} - \xi_{\pi_{1}^{(m)}(\beta_{0,m-1})}^{n,m} = e^{i\theta_{\beta_{1,m-1}}^{(m)}} \left( \gamma_{0, {\hat \pi}^{(m-1)}(d)}^{n,m} - \gamma_{0, {\hat \pi}^{(m-1)}(d)-1}^{n,m}  \right).
	\end{equation}
	directly from the definition of $\xi_{j}^{n,m}$ with $j=\pi_{1}^{(m)}(\beta_{0,m})$.
	

	From \eqref{thetan} we can write
	\begin{equation}\label{A4}
	\theta_{\alpha}^{(m)}=\left\{ 
		\begin{array}{ll}\vspace*{0.2cm}
		\theta_{\alpha}^{(m-1)}, & \alpha \neq \beta_{1,m-1},\\
		\theta_{\beta_{1,m-1}}^{(m-1)}+ \theta_{\beta_{0,m-1}}^{(m-1)}, & \alpha = \beta_{1,m-1},
		\end{array}	\right.
	\end{equation}	
	
Now, since for any $j <d$ we have $\gamma_{0,j}^{n,m-1}=\gamma_{0,j}^{n,m}$, from the above relations using \eqref{hatTnm} we prove our lemma in this case.


	Now assume  $\varepsilon(m-1)=1$. It is cleat that 
	$\pi_{1}^{(m-1)}= \pi_{1}^{(m)}$
	and
	$\pi_{0}^{(m)}(\beta_{1,m-1})= \pi_{0}^{(m)}(\beta_{0,m-1})-1.$
	With $j = {\tilde \pi}^{(m-1)}(d)-1$,  it is straightforward from the definition of $\xi_j^{n,m}$ that	
	\begin{equation}\label{A6}
	\xi_{{\tilde \pi}^{(m-1)}(d)-1}^{n,m} = e^{i\theta_{\beta_{0,m-1}}^{(m)}} \left( \gamma_{0,\pi_{0}^{(m)}(\beta_{1,m-1})}^{n,m} - \gamma_{0,\pi_{0}^{(m)}(\beta_{0,m-1})}^{n,m}  \right) + \xi_{{\tilde \pi}^{(m-1)}(d)}^{n,m}.
	\end{equation}
	
	By \eqref{eqiet9} and \eqref{x0x1} we have
	\begin{equation}\label{A6a}
	\gamma_{0,j}^{n,m-1}=  \left\{
	\begin{array}{ll}\vspace*{0.2cm}
	\gamma_{0,j}^{n,m} , & 0\leq j < {\hat \pi}^{(m)}(d),\\
	\gamma_{0,j+1}^{n,m} , & {\hat \pi}^{(m)}(d) \leq j < d,
	\end{array}\right.
	\end{equation}
	which in particular, by \eqref{xis} gives $\gamma_{0,d-1}^{n,m-1}=\xi_{d}^{n,m}$. Also, by \eqref{thetan} we have
	\begin{equation}\label{A7}
	\theta_{\alpha}^{(m)}=\left\{ 
	\begin{array}{ll}\vspace*{0.2cm}
	\theta_{\alpha}^{(m-1)}, & \alpha \neq \beta_{0,m-1},\\
	\theta_{\beta_{0,m-1}}^{(m-1)}+ \theta_{\beta_{1,m-1}}^{(m-1)}, & \alpha = \beta_{0,m-1},
	\end{array}	\right.
	\end{equation}

The second statement in the lemma follows from combining this with \eqref{A6} using 	the definition of ${\hat T}_{\alpha}^{n,m}$.

	\end{proof}

Before proving our next lemma, note that we can write \eqref{eqiet6}  as
		\begin{equation}\label{xi01}
		{\hat \pi}^{(m-1)}(j)= \left\{\begin{array}{ll}\vspace*{0.2cm}
		{\hat \pi}^{(m)}({ \tilde \pi  }^{(m)}(d)+1), & j=d,\\\vspace*{0.2cm}
		{\hat \pi}^{(m)}(j+1), & {\tilde \pi}^{(m-1)}(d)< j < d,\\
		{\hat \pi}^{(m)}(j), & j \leq {\tilde \pi}^{(m)}(d).
		\end{array}\right.
		\end{equation}

Our next two lemmas serve as technical tools for our next section. Most of their proofs consists of simple computations using our formulas and definitions.
We highlight the main steps but do not present an exhaustive proof.
	\begin{lemma}\label{LXi0}
		Let $n\geq 1$ and $0< m \leq n$. If $\varepsilon(m-1)=0$ and $\xi_{d-1}^{n,m-1}=\gamma_{1,d-1}^{n,m-1}$, then 
		\begin{equation}\label{xi02'}
		\hat{T}_{\alpha}^{(n,m-1)}(z)=\hat{T}_{\alpha}^{(n,m)}(z).
		\end{equation}
		for all $z \in \C$ and $\alpha \in \mathcal{A}\backslash \{ \beta_{1,m-1}  \}$.
	\end{lemma}

	\begin{proof}

		For  ${\tilde \pi}^{(m)}(d)< j < d$, from the definition of $\hat{\pi}^{(m-1)}$ and since $\gamma_{0,j}^{n,m-1}= \gamma_{0,j}^{n,m}$ we can write
		\begin{equation*}\label{xi03}
		 \gamma_{0,{\hat \pi}^{(m-1)}(j)}^{n,m-1} - \gamma_{0,{\hat \pi}^{(m-1)}(j)}^{n,m-1} =  \gamma_{0,{\hat \pi}^{(m)}(j+1)-1}^{n,m} - \gamma_{0,{\hat \pi}^{(m)}(j+1) }^{n,m}.
		\end{equation*}
		
		Since $\pi_{0}^{(m-1)}= \pi_{0}^{(m)}$,   $\left( \pi_{1}^{(m-1)}\right)^{-1}(j) = \left( \pi_{1}^{(m)}\right)^{-1}(j+1)$, and as $j<d$, using \eqref{A4} we get
		\begin{equation*}\label{xi04}
		\theta_{\left( \pi_{1}^{(m-1)}\right)^{-1}(j)}^{(m-1)} = \theta_{\left( \pi_{1}^{(m)}\right)^{-1}(j+1)}^{(m)}.
		\end{equation*}
		As $\xi_{d-1}^{n,m-1}= \gamma_{1,d-1}^{n,m-1}$ and $\gamma_{1,d-1}^{n,m-1}=\gamma_{0,d}^{n,m}$, the two expressions above give for ${\tilde \pi}^{(m)}(d)\leq j < d$ 
		\begin{equation}\label{xi05}
		\xi_{j}^{n,m-1}= \xi_{j+1}^{n,m}.
		\end{equation}
		
		Now assume $\alpha \in \mathcal{A}$ is such that $\pi_{1}^{(m)}(\alpha)>{\tilde \pi}^{(m)}(d)+1$. By \eqref{xi05} we get $\xi_{\pi_{1}^{(m)}(\alpha)}^{n,m}=\xi_{\pi_{1}^{(m)}(\alpha)-1}^{n,m-1}$, and since by \eqref{eqiet6}, we have $\pi_{1}^{(m-1)}(\alpha)= \pi_{1}^{(m)}(\alpha)-1$, this gives $$\xi_{\pi_{1}^{(m)}(\alpha)}^{n,m}=\xi_{\pi_{1}^{(m-1)}(\alpha)}^{n,m-1}.$$
		
Since $\gamma_{0,j}^{n,m-1}=	\gamma_{0,j}^{n,m}$ the proof of the lemma in this case follows from  the definition of $\xi_j^{n,m}$ and ${\hat T}_{\alpha}^{n,m}$.

		Note that $\pi_1^{(m-1)}(\beta_{0,m-1})=\pi_1^{(m)}(\beta_{1,m-1})$ and thus it follows from \eqref{xi05} that  $  \xi_{\pi_{1}^{(m-1)}(\beta_{0,m-1})}^{n,m-1}  =  \xi_{\pi_{1}^{(m)}(\beta_{1,m-1})}^{n,m} $.  By \eqref{A3} we get
		\begin{equation}\label{xi05c}
		\xi_{\pi_{1}^{(m-1)}(\beta_{0,m-1})}^{n,m-1} - \xi_{\pi_{1}^{(m)}(\beta_{0,m-1})}^{n,m} = e^{i\theta_{\beta_{1,m-1}}^{(m)}} \left( \gamma_{0, {\hat \pi}^{(m-1)}(d)}^{n,m} - \gamma_{0, {\hat \pi}^{(m-1)}(d)-1}^{n,m}  \right).
		\end{equation}
		
		Since $\xi_{d-1}^{n,m-1}=\gamma_{1,d-1}^{n,m-1}$, we have
		\begin{equation}\label{xi01'}
		\gamma_{0,d}^{n,m-1} - \gamma_{1,d-1}^{n,m-1} = e^{i\theta_{\beta_{1,m-1}}^{(m-1)}} \left( \gamma_{0, {\hat \pi}^{(m-1)}(d)}^{n,m-1} - \gamma_{0, {\hat \pi}^{(m-1)}(d)-1}^{n,m-1}  \right),
		\end{equation}
		which combined with \eqref{A4} and \eqref{xi05c}, using the fact that $\gamma_{1,d-1}^{n,m-1} = \gamma_{0,d}^{n,m} $, $\gamma_{0,j}^{n,m-1}=\gamma_{0,j}^{n,m}$, when $j<d$ and the definition of $\xi_j^{n,m}$ proves the lemma  for $\alpha= \beta_{0,m-1}$.\\


	From \eqref{A4}, \eqref{xi01} and \eqref{xi01'},  as $\pi_{1}^{(m)}(\beta_{1,m-1})= \pi_{1}^{(m)}(\beta_{0,m})+1$
and $\gamma_{0,d}^{n,m}=\gamma_{1,d-1}^{n,m-1}$, a trivial computation gives
		\begin{equation*}\label{xi09}
		\xi_{{ \tilde \pi  }^{(m)}(d)+1}^{n,m} = e^{i\theta_{\beta_{0,m}}^{(m)}} \left( \gamma_{0, d}^{n,m-1} - \gamma_{0, d}^{n,m}  \right) + \xi_{{ \tilde \pi  }^{(m)}(d)}^{n,m} .
		\end{equation*}
		
		By \eqref{xi01}, \eqref{xi05} and from the definition of $\xi_{{ \tilde \pi  }^{(m)}(d)-1}^{n,m-1}$ we get
		\begin{equation*}\label{}
				\xi_{{ \tilde \pi  }^{(m)}(d)-1}^{n,m-1} = e^{i\theta_{\beta_{0,m}}^{(m-1)}} \left( \gamma_{0, d-1}^{n,m-1} - \gamma_{0, d}^{n,m-1}  \right) + \xi_{{ \tilde \pi  }^{(m)}(d)+1}^{n,m}.
		\end{equation*}
		Combining this with \eqref{A4}, \eqref{xi01} and noting that $\gamma_{0,d-1}^{n,m} = \gamma_{0,d-1}^{n,m-1} $, the relation
		\begin{equation*}
			\xi_{{ \tilde \pi  }^{(m)}(d)-1}^{n,m-1} =	\xi_{{ \tilde \pi  }^{(m)}(d)-1}^{n,m}.
		\end{equation*}		
		simply follows from the definition of $\xi_j^{n,m}$ for $j=\tpi^{(m)}(d)-1$.

	We now prove by induction on $j$ that		
		\begin{equation}\label{xi013}
		\xi_{j}^{n,m-1} =	\xi_{j}^{n,m}.
		\end{equation}
		for $1 \leq j < \tpi^{(m)}(d)$. 
		
		Since $\pi_{0}^{(m-1)}= \pi_{0}^{(m)}$, we get by   \eqref{xi01} that $( \pi_{1}^{(m-1)})^{-1}(j) = ( \pi_{1}^{(m)})^{-1}(j)$, and as $j<d$, by \eqref{A4} we have
		\begin{equation}\label{xi012}
		\theta_{\left( \pi_{1}^{(m-1)}\right)^{-1}(j)}^{(m-1)} = \theta_{\left( \pi_{1}^{(m)}\right)^{-1}(j)}^{(m)},
		\end{equation}
		Combined with \eqref{xi01} this gives
		$$\xi_{j-1}^{n,m-1}= e^{i \theta^{(m)}_{ (\pi_1^{(m)})^{-1} (j)  }}   \left( \gamma^{n,m}_{0,  {\hat \pi}^{(m)}(j) -1  }   -  \gamma^{n,m}_{0,  {\hat \pi}^{(m)}(j) }   \right)+\xi_{j}^{n,m},$$
		which, as $\xi_{j}^{n,m-1} =	\xi_{j}^{n,m}$, by \eqref{xis} shows that $\xi_{j-1}^{n,m-1} =	\xi_{j-1}^{n,m}$, proving \eqref{xi013}.

		Now assume $\alpha \in \mathcal{A}$ is such that $\pi_{1}^{(m)}(\alpha)< \tpi^{(m)}(d)$.
		From \eqref{xi013}, since $\pi_1^{(m-1)}(\alpha)= \pi_1^{(m)}(\alpha)$ we get $\xi_{\pi_1^{(m-1)}(\alpha)}^{n,m-1} =	\xi_{\pi_1^{(m)}(\alpha)}^{n,m}$. This, combined with  \eqref{xi012} and the definition of $\hat{T}_{\alpha}^{(n,m)}$, proves our statement  for $\pi_{1}^{(m)}(\alpha)< \tpi^{(m)}(d)$.

		Since $\pi_1^{(m)}(\beta_{0,m-1})= \tpi^{(m)}(d)$ and since we proved \eqref{xi02'}  for $\alpha = \beta_{0,m-1}$ and $\pi_{1}^{(m)}(\alpha)> \tpi^{(m)}(d)+1$, we get \eqref{xi02'}, for all $\pi_{1}^{(m)}(\alpha)\neq \tpi^{(m)}(d)+1$. As $\pi_{1}^{(m)}(\beta_{1,m-1})= \pi_{1}^{(m)}(\beta_{0,m})+1$ and   $\pi_{1}^{(m)}(\beta_{1,m})=\tpi^{(m)}(d)+1$ we have \eqref{xi02'} for all $\alpha \in \mathcal{A}\backslash \{ \beta_{1,m-1}  \}$.		
	\end{proof}

Note that by \eqref{eqiet8} we can write
		\begin{equation}\label{xi11c}
		{\hat \pi}^{(m-1)}(j)= \left\{\begin{array}{ll}\vspace*{0.2cm}
		{\hat \pi}^{(m)}(j)-1, & \hpi^{(m)}(j)>\hpi^{(m)}(d)+1,\\\vspace*{0.2cm}
		d, & \hpi^{(m)}(j)= \hpi^{(m)}(d)+1,\\
		{\hat \pi}^{(m)}(j), &  \hpi^{(m)}(j)<\hpi^{(m)}(d)+1.
		\end{array}\right.
		\end{equation}

	The following lemma provides a result similar to that of Lemma \ref{LXi0}, but for the case $\varepsilon(m-1)=1$. The main difference, compared to the previous case, comes from the fact that $\xi_{d-1}^{n,m-1}$ does not, beforehand, coincide with $\gamma_{1,d-1}^{n,m-1}$, although we will later establish this equality.

	\begin{lemma}\label{LXi1}
	Let $n\geq 1$ and $0< m \leq n$. If $\varepsilon(m-1)=1$ and $\xi_{d-1}^{n,m-1}=\xi_{d-1}^{n,m}$, then for all $z \in \C$ and $\alpha \in \mathcal{A}\backslash \{ \beta_{0,m-1}, \beta_{1,m-1}  \}$  we have
	\begin{equation}\label{xi11'}
	\hat{T}_{\alpha}^{(n,m-1)}(z)=\hat{T}_{\alpha}^{(n,m)}(z).
	\end{equation}
and
	\begin{equation}\label{xi12'}
	\hat{T}_{\beta_{0,m-1}}^{(n,m-1)}(z)=\hat{T}_{\beta_{0,m-1}}^{(n,m)} \circ \left(\hat{T}_{\beta_{1,m-1}}^{(n,m)}\right)^{-1}(z).
	\end{equation}
    \end{lemma}

	\begin{proof}

		By  \eqref{A6a} and \eqref{xi11c},   for all $j$ such that $\hpi^{(m)}(j)\notin \{ \hpi^{(m)}(d), \hpi^{(m)}(d)+1    \}$, we get
		\begin{equation}\label{xi11}
		\gamma_{0,\hpi^{(m-1)}(j)}^{n,m-1}= \gamma_{0,\hpi^{(m)}(j)}^{n,m},
		\end{equation}
		similarly, $\gamma_{0,\hpi^{(m-1)}(j)-1}^{n,m-1}= \gamma_{0,\hpi^{(m)}(j)-1}^{n,m}$. In particular, for any $j \notin \{ \tpi^{(m)}(\hpi^{(m)}(d)+1), d    \}$ we have
		\begin{equation}\label{xi12}
		\gamma_{0,\hpi^{(m-1)}(j)-1}^{n,m-1}-\gamma_{0,\hpi^{(m-1)}(j)}^{n,m-1} = \gamma_{0,\hpi^{(m)}(j)-1}^{n,m}-\gamma_{0,\hpi^{(m)}(j)}^{n,m}.
		\end{equation}
		
		As $\pi_1^{(m-1)}=\pi_1^{(m)}$ and by \eqref{A7}, for all $j<d$ we have
		\begin{equation}\label{xi13}
		\theta_{(\pi_1^{(m-1)})^{-1}(j)}^{(m-1)} = \theta_{(\pi_1^{(m)})^{-1}(j)}^{(m)}.
		\end{equation}

	We now prove, by induction on $j$, that
			\begin{equation}\label{xi14}
		\xi_{j}^{n,m-1}= \xi_{j}^{n,m}.
		\end{equation}
		for $\tpi^{(m-1)}(d)\leq j<d$.
		
		We have $\xi_{d-1}^{n,m-1}= \xi_{d-1}^{n,m}$. Take $\tpi^{(m-1)}(d)< j<d$. As $\tpi^{(m)}(\hpi^{(m)}(d)+1) = \tpi^{(m-1)}(d)$, we have that $ j \notin \{\tpi^{(m)}(\hpi^{(m)}(d)+1), d    \}$, hence by \eqref{xi12} and \eqref{xi13} we get $\xi_{j-1}^{n,m-1}= \xi_{j-1}^{n,m}$. This shows that for any $\tpi^{(m-1)}(d) \leq j<d$, \eqref{xi14} holds.

		By \eqref{A6a} and \eqref{xi11c} we have $\gamma_{0,\hpi^{(m-1)}(d)}^{n,m-1}=\gamma_{0,\hpi^{(m)}(d)+1}^{n,m}$ and $\gamma_{0,\hpi^{(m-1)}(d)-1}^{n,m-1}=\gamma_{0,\hpi^{(m)}(d)-1}^{n,m}$, thus by \eqref{xis} we get
		$$  \xi_{d-1}^{n,m-1}= e^{i \theta^{(m-1)}_{ \beta_{1,m-1}  }}   \left( \gamma^{n,m}_{0,  {\hat \pi}^{(m)}(d) -1  }   -  \gamma^{n,m}_{0,  {\hat \pi}^{(m)}(d)+1 }   \right)+\gamma_{0,d}^{n,m-1}, $$
		since $\xi_{d-1}^{n,m-1} =  \xi_{d-1}^{n,m}  $ and $\xi_{d}^{n,m}=\gamma_{0,d-1}^{n,m-1}$, by combining this with the definition of $\xi_{d-1}^{n,m}$ and \eqref{xi13} we get
		\begin{equation}\label{xi16}
		\gamma_{0,d-1}^{n,m-1}- \gamma_{0,d}^{n,m-1} = e^{i \theta_{\beta_{1,m-1}}^{(m-1)}} \left( \gamma^{n,m}_{0,  {\hat \pi}^{(m)}(d) }   -  \gamma^{n,m}_{0,  {\hat \pi}^{(m)}(d)+1 }   \right).
		\end{equation}
		
		By \eqref{xis}, with $j= \tpi^{(m-1)}(d) -1$, we have
		$$  \xi_{\tpi^{(m-1)}(d)-1}^{n,m-1}   =  e^{i \theta_{\beta_{0,m-1}}^{(m-1)}} \left( \gamma^{n,m-1}_{0,  d-1 }   -  \gamma^{n,m-1}_{0,  d }   \right)  +   \xi_{\tpi^{(m-1)}(d)}^{n,m-1},  $$
		which by and \eqref{A7}, \eqref{xi14} and \eqref{xi16}  gives
		$$  \xi_{\tpi^{(m-1)}(d)-1}^{n,m-1}   =  e^{i \theta_{\beta_{0,m-1}}^{(m)}} \left( \gamma^{n,m}_{0,  {\hat \pi}^{(m)}(d) }   -  \gamma^{n,m}_{0,  {\hat \pi}^{(m)}(d)+1 }   \right)  +   \xi_{\tpi^{(m-1)}(d)}^{n,m},  $$
		and as, by \eqref{xi11c}, $\tpi^{(m-1)}(d)= \tpi^{(m)}(\hpi^{(m)}(d)+1)$, combined with \eqref{xis} this shows that $\xi_{\tpi^{(m-1)}(d)-1}^{n,m-1}  = \xi_{\tpi^{(m-1)}(d)-1}^{n,m}$.
		
		Now assume, by induction in $j$, that for some $j<\tpi^{(m-1)}(d)$ we have $\xi_{j}^{n,m-1} = \xi_{j}^{n,m}    $. It is straightforward to see, by definition of $\xi_{j}^{n,m}$, \eqref{xi12} and \eqref{xi13} that $\xi_{j-1}^{n,m-1} = \xi_{j-1}^{n,m}    $.
		Since we had proved before that \eqref{xi14} holds for $\tpi^{(m-1)}(d) \leq j < d$, this shows that \eqref{xi14} is true for all $j<d$.
		
		Now, consider $\alpha \in \mathcal{A} \backslash \{  \beta_{0,m-1},\beta_{1,m-1}  \}$. By taking $j=\pi_{1}^{(m)}(\alpha)$ we get $j \notin \{ \tpi^{(m)}(\hpi^{(m)}(d)+1), d    \}$ and by \eqref{xi11} we obtain
		$ \gamma_{0,\pi_{0}^{(m-1)}(\alpha)}^{n,m-1}= \gamma_{0,\pi_{0}^{(m)}(\alpha)}^{n,m} , $ and thus by \eqref{xi13}, \eqref{xi14} and \eqref{hatTnm} we get \eqref{xi11'}.\\
		
		By \eqref{hatTnm}, for all $z \in \C$, we get
		\begin{equation*}\label{}
		\begin{array}{l}\vspace*{0.2cm}
		\hat{T}_{\beta_{0,m-1}}^{(n,m)} \circ \left(\hat{T}_{\beta_{1,m-1}}^{(n,m)}\right)^{-1}(z) = \xi_{\pi_1^{(m)}(\beta_{0,m-1})}^{n,m} +\\
		e^{i \theta_{ \beta_{0,m-1} }^{ (m) }}\left[    e^{- i \theta_{\beta_{1,m-1}}^{(m)} }  \left( z - \xi_{\pi_1^{(m)}(\beta_{1,m-1})}^{n,m}         \right)  + \gamma_{0,\pi_0^{(m)}(\beta_{1,m-1})}^{n,m}   -   \gamma_{0,\pi_0^{(m)}(\beta_{0,m-1})}^{n,m} \right] ,
		\end{array}
		\end{equation*}
		which by Lemma \ref{LA}  gives
		$$ \hat{T}_{\beta_{0,m-1}}^{(n,m)} \circ \left(\hat{T}_{\beta_{1,m-1}}^{(n,m)}\right)^{-1}(z) =  e^{i\theta_{\beta_{0,m-1}}^{(m-1)}}\left( z -  \gamma_{0,d-1}^{n,m-1}   \right)  + \xi_{0,{\tilde  \pi}^{(m-1)}(d)-1}^{n,m}    ,$$
		 combined with \eqref{xi14} and \eqref{xis} for $j= {\tilde  \pi}^{(m-1)}(d)-1 $, this gives
		$$ \hat{T}_{\beta_{0,m-1}}^{(n,m)} \circ \left(\hat{T}_{\beta_{1,m-1}}^{(n,m)}\right)^{-1}(z) = e^{i\theta_{\beta_{0,m-1}}^{(m-1)}}\left( z -  \gamma_{0,d}^{n,m-1}   \right)  + \xi_{0,{\tilde  \pi}^{(m-1)}(d)}^{n,m} .$$
		By \eqref{hatTnm} this shows that \eqref{xi12'} holds.		
	\end{proof}

\vspace*{0.2in}

Consider now $J=\{ J_k  \}_{0\leq k<r}$, with $r \in \N $, an ordered sequence of disjoint subintervals of $I$. Let $I'$ be a subinterval of $I$, we denote
$$  J \cap I' = \{ J_k\cap I' : J_k \cap I' \neq \emptyset    \}_{0\leq k<r}.  $$
Recall we denote by $J^{(n+1)}$ the ordering of $\left\{   f_{\lambda,\pi}^k\left(I^{(n)}\backslash I^{(n+1)}\right)   \right\}_{0\leq k<r(n)}$ ,
where $r(n) = r_{\lambda,\pi}^{n}\left(I_{\beta_{0,n}}^{(n)}\right)$.

Given $n \in \N$ we define a sequence $\{k(m)\}_{0\leq m \leq n+1}$ of indices of $J^{(n+1)}$ as follows. Set $k(n+1)=0$. For $0 \leq m < n+1$ let $k(m)$ be equal to the number of disjoint subintervals in $J^{(n+1)}\cap I^{(m)}$. It is clear we have
$$  J^{(n+1)} \cap (I^{(m-1)}\backslash I^{(m)}) = \{ J_{k} \}_{k(m)\leq k < k(m-1)}, $$
for $0 < m \leq n+1$.

Denote by $\beta(m) = \beta_{1-\varepsilon(m),m}$ that is, the loser of $(\lambda^{(m)},\pi^{(m)})$. The following two lemmas, which describe $J^{(n+1)}$, will be needed for the following section. 
	
	\begin{lemma}\label{LIET2}
		For all $n \geq 0$, $0<m\leq n+1$ and $0 \leq k <r(m-1)$, if $J_k \cap (I^{(m-1)} \backslash I^{(m)}) \neq \emptyset  $ then $J_k \subseteq I^{(m-1)} \backslash I^{(m)}$.
	\end{lemma}
	
	\begin{proof}
		Assume, by contradiction, that there is a $J_k=J_k' \sqcup J_k'' \in J^{(n+1)}$ such that $J_k'\cap (I^{(m-1)}\backslash I^{(m)})= \emptyset$ and $J_k'' \subseteq I^{(m-1)}\backslash I^{(m)}$. Take $l \geq 0$ such that $f_{\lambda,\pi}^{-l}(J_k') \subseteq I^{(n)}\backslash I^{(n+1)} $. It is simple to check, given two points $x' \in f_{\lambda,\pi}^{-l}(J_k')$ and $x'' \in I_{\beta_{0,n}}^{(n)} \backslash f_{\lambda,\pi}^{-l}(J_k')$, that $r_{\lambda,\pi}^{n} (x')\neq r_{\lambda,\pi}^{n} (x'')$, which, as $I^{(n)}\backslash I^{(n+1)} \subseteq I_{\beta_{0,n}}^{(n)}$  contradicts the fact that  $r_{\lambda,\pi}^{n}$ is constant on $I_{\beta_{0,n}}^{(n)}$.
	\end{proof}

	\begin{lemma}\label{LIET1}
		For all $n \geq 0$ we have
		$$ J^{(n+1)} \cap (I^{(n)} \backslash I^{(n+1)}) = \{  I^{(n)}\backslash I^{(n+1)}   \}, $$
		furthermore for all $0<m\leq n$ we have
		\begin{equation}\label{LIET1eq1}
		J^{(n+1)} \cap (I^{(m-1)} \backslash I^{(m)}) = f_{\lambda^{(m-1)},\pi^{(m-1)}}\left(  J^{(n+1)}\cap I_{\beta(m-1)}^{(m)}    \right)  .
		\end{equation}
		In particular there exists a $k'(m)>0$ such that for all $k(m)  \leq k < k(m-1) $ we have
		\begin{equation}\label{LIET1eq2}
		J_k=  f_{\lambda^{(m-1)},\pi^{(m-1)}}\left(  J_{k-k'(m)}     \right) .
		\end{equation}
	\end{lemma}
	
	\begin{proof}
		Note that we have $J_0 = I^{(n)}\backslash I^{(n+1)}$ from whence the first statement follows.
		
		Assume that $J^{(n+1)}\cap I^{(m-1)}\backslash I^{(m)} \neq \emptyset$, as otherwise the result holds trivially, and take $J_k \in J^{(n+1)}$ such that $J_k \cap (I^{(m-1)}\backslash I^{(m)}) \neq \emptyset$. By Lemma \ref{LIET2} we have $J_k \subseteq I^{(m-1)} \backslash  I^{(m)}$ and thus  it follows from the definition of $J^{(n+1)}$ that there is an $l\geq1$ such that $f_{\lambda^{(m-1)},\pi^{(m-1)}}^{l}(I^{(n)}\backslash I^{(n+1)})=J_k$. Furthermore the pre-image by $f_{\lambda^{(m-1)},\pi^{(m-1)}}$ of $J_k$ is contained in $I_{\beta(m-1)}^{(m)}$ and it is a term $J_{k'}$, with $k'<k$, in the sequence $J^{(n+1)}$. The difference $k'(m)=k-k'$ is independent of the choice of $J_k$, from which \eqref{LIET1eq2} follows.
		Observing that $J^{(n+1)}\cap I_{\beta(m-1)}^{(m)} = \{J_k\}_{k \in K}$, with $K=\{  k(m)-k'(m), ..., k(m-1)-k'(m)   \}$,  and combining this with \eqref{LIET1eq2} we obtain \eqref{LIET1eq1}, thus completing the proof.
	\end{proof}

\section{Existence of a quasi-embedding}

In this section we introduce the notion of \textit{quasi-embedding} and use it to relate the dynamics of $f_{\lambda^{(m)},\pi^{(m)}}$ with that of $T^{(n,m)}$ for any $n\geq 0$ and $0\leq m\leq n$.

We say that $f_{\lambda^{(m)},\pi^{(m)}}$ is \textit{quasi-embedded} into $T^{(n,m)}$, or that $\gamma^{(n)}$ is a \textit{quasi-embedding} of $f_{\lambda^{(m)},\pi^{(m)}}$ into $T^{(n,m)}$, for $x \in I' \subseteq I$ if
\begin{equation*}\label{X1'}
T^{(n,m)}(\gamma^{(n)}(x))= \gamma^{(n)}(f_{\lambda^{(m)},\pi^{(m)}}(x)).
\end{equation*}
Intuitively this means that $T^{(n,m)}$ and $f_{\lambda^{(m)},\pi^{(m)}}$ are \textit{nearly} topologically conjugate, the conjugacy failing only for points in  $I \backslash I'$.

The following theorem establishes that $T_{\alpha}^{(n,m)} = 	{\hat T}_{\alpha}^{(n,m)}$ and that  $\gamma^{(n)}$ is a quasi-embedding of $f_{\lambda^{(m)},\pi^{(m)}}$ into $T^{(n,m)}$ except for points in a subinterval which decreases with  $n$.
	\begin{theorem}\label{fund_lemma}
	For all  $n\geq 0$ and  $0\leq m \leq n$, $\gamma^{(n)}$ is a quasi-embedding of $f_{\lambda^{(m)},\pi^{(m)}}$ into $T^{(n,m)}$ for $x \in I^{(m)}\backslash f_{\lambda^{(m)},\pi^{(m)}}^{-1}\left(  I^{(n)}     \right) $. Furthermore for all $\alpha \in \mathcal{A}$ and $z \in \C$ we have
	\begin{equation}\label{X2'}
		T_{\alpha}^{(n,m)}(z) = {\hat T}_{\alpha}^{(n,m)}(z).
	\end{equation}
	\end{theorem}
The remainder of this section is reserved for the proof of several lemmas culminating in the proof of Theorem \ref{fund_lemma}.

\vspace*{0.2in}


The first lemma is a particular case of Theorem \ref{fund_lemma} where $n \geq 1$ and $m=n-1$. We separate the cases $\varepsilon(m-1)=0,1$ as $T_{\alpha}^{(n,m)}$ is given by different expressions in each. For the case $\varepsilon(m-1)=0$ we use Lemma \ref{LA} to obtain an expression for $T_{\beta_{1,n-1}}^{(n,n-1)}$ and distinguish between the cases $\alpha=\beta_{1,n-1}$ and $\alpha \neq \beta_{1,n-1}$, as the first can be addressed directly while the second requires the use of Lemma \ref{LXi0}. In the case $\varepsilon(m-1)=1$ we also distinguish between $\alpha=\beta_{1,n-1}$ and $\alpha \neq \beta_{1,n-1}$ as, unlike the first, the latter case requires the use of Lemma \ref{LXi1}.

\begin{lemma}\label{lfund_lemma1}
	Let  $n\geq 1$ and $\alpha \in \mathcal{A}$. Then $\gamma^{(n)}$ is a quasi-embedding of $f_{\lambda^{(n-1)},\pi^{(n-1)}}$ into $T^{(n,n-1)}$ for $x \in I^{(n-1)}\backslash f_{\lambda^{(n-1)},\pi^{(n-1)}}^{-1}\left(  I^{(n)}     \right) $. Furthermore for all $z \in \C$ we have	
	\begin{equation}\label{X2'l}
	T_{\alpha}^{(n,n-1)}(z) = 	{\hat T}_{\alpha}^{(n,n-1)}(z).
	\end{equation}
\end{lemma}

\begin{proof}
	
We distinguish the cases $\varepsilon(n-1)=0$ and $\varepsilon(n-1)=1$.

Given $n\geq 1$ assume  $\varepsilon(n-1)=0$ . 
Lemma \ref{LA} for $m=n$ combined with \eqref{Tnm0} gives
\begin{equation}\label{e4}
T_{\beta_{1,n-1}}^{(n,n-1)} (z)= e^{i\theta_{\beta_{1,n-1}}^{(n-1)}}\left( z - \gamma_{0,{\hat  \pi}^{(n-1)}(d)-1}^{n,n-1}    \right)   +  \gamma_{1,d-1}^{n,n-1}  .
\end{equation} 
for all $z \in \C$.

By Lemma \ref{LIET1}, $J^{(n)} =\{ I^{(n-1)} \backslash I^{(n)} \}$.
Let $x \in I^{(n-1)}_{\beta_{1,n-1}}\backslash f_{\lambda^{(n-1)},\pi^{(n-1)}}^{-1}(I^{(n)})$. Since we have $ f_{\lambda^{(n-1)},\pi^{(n-1)}}(x) \in I^{(n-1)} \backslash I^{(n)}$, it follows from our definitions of breaking operator and breaking sequence that
\begin{equation}\label{e4a}
\gamma^{(n)}(f_{\lambda^{(n-1)},\pi^{(n-1)}}(x)) = f_{\lambda^{(n-1)},\pi^{(n-1)}}(x) e^{i {\theta}_{\beta_{1,n-1}}^{(n-1)}} + |I^{(n)}|(1-e^{i {\theta}_{\beta_{1,n-1}}^{(n-1)}}),
\end{equation}
since $x_{1,d-1}^{(n-1)}= |I^{(n)}|$,  $\gamma_{0, \hpi^{(n-1)}(d)-1}^{n,n-1} = x_{0, \hpi^{(n-1)}(d)-1}^{(n-1)}$ as $\gamma^{(n)}(x)=x$ and $|I^{(n)}|=\gamma_{1,d-1}^{n,n-1}$, this gives
\begin{equation}\label{e5}
\gamma^{(n)}(f_{\lambda^{(n-1)},\pi^{(n-1)}}(x))-e^{i {\theta}_{\beta_{1,n-1}}^{(n-1)}} \gamma^{(n)}(x) = \gamma_{1,d-1}^{n,n-1} - e^{i {\theta}_{\beta_{1,n-1}}^{(n-1)}}\gamma_{0, \hpi^{(n-1)}(d)-1}^{n,n-1}.
\end{equation}
As $I_{\alpha}^{(n-1)}\backslash f_{\lambda^{(n-1)},\pi^{(n-1)}}^{-1}(I^{(n)}) = \emptyset$ for $\alpha \neq \beta_{1,n-1}$, \eqref{e4} together with \eqref{e5} gives
\begin{equation*}\label{X1'l}
	T^{(n,n-1)}(\gamma^{(n)}(x))= \gamma^{(n)}(f_{\lambda^{(n-1)},\pi^{(n-1)}}(x)),
\end{equation*}
which proves that the map $\gamma^{(n)}$ is a quasi-embedding of $f_{\lambda^{(n-1)},\pi^{(n-1)}}$ into $T^{(n,n-1)}$ for $x \in I_{\alpha}^{(n-1)}\backslash f_{\lambda^{(n-1)},\pi^{(n-1)}}^{-1}\left(  I^{(n)}     \right) $.

By continuity of $f_{\lambda^{(n-1)},\pi^{(n-1)}}$ in $I_{\beta_{1,n-1}}^{(n-1)}= [x_{0,\hpi^{(n-1)}(d)-1}^{(n-1)},x_{0,\hpi^{(n-1)}(d)}^{(n-1)})$ and from \eqref{e5} we get
\begin{equation*}
\gamma_{1,d-1}^{n,n-1} - e^{i {\theta}_{\beta_{1,n-1}}^{(n-1)}}\gamma_{0, \hpi^{(n-1)}(d)-1}^{n,n-1} = \gamma_{1,d}^{n,n-1} - e^{i {\theta}_{\beta_{1,n-1}}^{(n-1)}}\gamma_{0, \hpi^{(n-1)}(d)}^{n,n-1}, 
\end{equation*}
which combined with \eqref{e4}, \eqref{xis} and \eqref{hatTnm} gives \eqref{X2'l} for $\alpha = \beta_{1,n-1}$.

Since $\xi_{d-1}^{n,n-1}=\gamma_{1,d-1}^{n,n-1}$, by Lemma \ref{LXi0} we prove the second statement in our lemma for all $\alpha \in \mathcal{A}$.

Now assume $\varepsilon(n-1)=1$. 
It follows directly from our definitions of 
${\hat T}_{\alpha}^{(n,m)} (z)$, $T_{\alpha}^{(n,n)} (z)$ and  $\xi_j^{n,m}$ using \eqref{Tnm1} that
\begin{equation}\label{e7}
{T}_{\beta_{1,n-1}}^{(n,n-1)} (z)=  e^{ i \theta_{\beta_{1,n-1}}^{(n-1)} }   \left( z - \gamma_{0, \hpi^{(n)}(d)  }^{n,n}   \right)   +   \gamma_{ 1,d }^{n,n},
\end{equation}
for all $z \in \C$. Again, in this case we also have 
$J^{(n)}=\{  I^{(n-1)} \backslash I^{(n)}  \}$ and we can use \eqref{e4a} as before which since $\theta_{\beta_{1,n-1}}^{(n-1)}  =  \theta_{\beta_{1,n}}^{(n)}$, $\gamma_{1,d}^{n,n} = x_{1,d}^{(n)}$ and  $\gamma_{0,\hpi^{(n)}(d)}^{n,n}= x_{0,\hpi^{(n)}(d)}^{(n)}$,  gives
\begin{equation}\label{e8}
\gamma^{(n)}(f_{\lambda^{(n-1)},\pi^{(n-1)}}(x))-e^{i {\theta}_{\beta_{1,n-1}}^{(n-1)}} \gamma^{(n)}(x) = \gamma_{1,d}^{n,n} - e^{i {\theta}_{\beta_{1,n-1}}^{(n-1)}}\gamma_{0, \hpi^{(n)}(d)}^{n,n},
\end{equation}
for all $x \in  I_{\beta_{1,n-1}}^{(n-1)} \backslash f_{\lambda^{(n-1)},\pi^{(n-1)}}^{-1}(I^{(n)})$. 
As $I_{\alpha}^{(n-1)}\backslash f_{\lambda^{(n-1)},\pi^{(n-1)}}^{-1}(I^{(n)}) = \emptyset$ for $\alpha \neq \beta_{1,n-1}$, combining \eqref{e7} and \eqref{e8} we prove the first statement in the lemma.

By continuity of $f_{\lambda^{(n-1)},\pi^{(n-1)}}$ in $I_{\beta_{1,n-1}}^{(n-1)}= [x_{0,\hpi^{(n-1)}(d)-1}^{(n-1)},x_{0,\hpi^{(n-1)}(d)}^{n-1})$ and from \eqref{e8} we 
can relate the image by $\gamma^{(n)}$ of the $d$-th endpoint of the partitions associated to $f_{\lambda^{(n-1)}, \pi^{(n-1)}}$ and  $f_{\lambda^{(n)}, \pi^{(n)}}$ as follows

\begin{equation*}
\gamma_{1,d}^{n,n-1} - e^{i {\theta}_{\beta_{1,n-1}}^{(n-1)}}\gamma_{0, \hpi^{(n-1)}(d)}^{n,n-1} = \gamma_{1,d}^{n,n} - e^{i {\theta}_{\beta_{1,n-1}}^{(n-1)}}\gamma_{0, \hpi^{(n)}(d)}^{n,n}.
\end{equation*}
As $\gamma_{1,d}^{n,n-1}= \xi_{d}^{n,n-1}$, this together with \eqref{e7} and \eqref{hatTnm}, proves  \eqref{X2'l} for $\alpha = \beta_{1,n-1}$. Using the definition of $\xi_j^{n,m}$ this can be rewritten as
$$ \xi_{d-1}^{n,n-1}= e^{i {\theta}_{\beta_{1,n-1}}^{(n-1)}  }   \left( \gamma^{n,n-1}_{0,  {\hat \pi}^{(n-1)}(d) -1  }   -  \gamma^{n,n}_{0,  {\hat \pi}^{(n)}(d) }   \right)+\xi_{d}^{n,n}, $$
and since $ \gamma^{n,n-1}_{0,  {\hat \pi}^{(n-1)}(d) -1  }  = \gamma^{n,n}_{0,  {\hat \pi}^{(n)}(d) -1  }   $, by \eqref{xis} and \eqref{A7}  we get that $\xi_{d-1}^{n,n-1} = \xi_{d-1}^{n,n}$. 
Hence by Lemma \ref{LXi1}, \eqref{hatTnm}, \eqref{Tnn} and \eqref{Tnm1} we prove the second statement in the lemma for all $\alpha \in \mathcal{A}$.
	
\end{proof}



Recall we denote by $J^{(n+1)}$ the ordering of $\left\{   f_{\lambda,\pi}^k\left(I^{(n)}\backslash I^{(n+1)}\right)   \right\}_{0\leq k<r(n)}$ , where $r(n) = r_{\lambda,\pi}^{n}\left(I_{\beta_{0,n}}^{(n)}\right)$.
Given $0<m\leq n+1$, by Lemma \ref{LIET1} there exist $0< k(m)< k(m-1)$ such that
$$J^{(n+1)}\cap (I^{(m-1)}\backslash I^{(m)}) = \{ J_k  \}_{k(m) \leq k < k(m-1)}, $$
and there exists $k'(m)>0$ such that
$$  J_k= f_{\lambda^{(m-1)},\pi^{(m-1)}}(J_{k-k'(m)}). $$
In particular we have the following relations
\begin{equation}\label{C5'}
[ x_{0,d}^{(m)}, y_{k(m)}   ) = f_{\lambda^{(m-1)},\pi^{(m-1)}}\left( [  f_{\lambda^{(m-1)},\pi^{(m-1)}}^{-1}(x_{0,d}^{(m)})  , y_{k(m)-k'(m)} )           \right),
\end{equation}
\begin{equation}\label{C6'}
[  y_{k(m+1)-1} + \Delta  , x_{0,d}^{(m-1)}   ) = f_{\lambda^{(m-1)},\pi^{(m-1)}}\left( [  y_{k(m+1)-1-k'(m)}  , x_{0,{\hat \pi}^{(m-1)}(d)}^{(m-1)} )           \right),
\end{equation}
recalling we denote $ J_k = [  y_k, y_{k}+\Delta   )$, for all $k(m)\leq k < k(m+1)$ we have
\begin{equation}\label{C7'}
J_k = f_{\lambda^{(m-1)},\pi^{(m-1)}}\left([ y_{k(m)-k'(m)} , y_{k(m)-k'(m)} + \Delta )           \right),
\end{equation}
and denoting $J_k'=[  y_k + \Delta, y_{k+1}  )$, for all $k(m)\leq k < k(m+1)-1$ we have
\begin{equation}\label{C8'}
J_k' = f_{\lambda^{(m-1)},\pi^{(m-1)}}\left([  y_{k(m)-k'(m)} + \Delta, y_{k(m)+1-k'(m)} )           \right).
\end{equation}


With the assumptions that $T_{\alpha}^{(n,m-1)} ={ \hat T}_{\alpha}^{(n,m-1)}$  and that $\gamma^{(n)}$ and $\gamma^{(n+1)}$ are quasi-embeddings of $f_{\lambda^{(m-1)},\pi^{(m-1)}}$ respectively into $T^{(n,m-1)}$ for all $x \in I^{(m-1)}\backslash f_{\lambda^{(m-1)},\pi^{(m-1)}}^{-1}\left(  I^{(n)}     \right) $ and into $T^{(n+1,m-1)}$ for some point in $I_{\beta_{1,m-1}}^{(m)}$, Lemmas \ref{LKa} and \ref{LKb} provide a means to extend the quasi-embedding $\gamma^{(n+1)}$ for points in a larger subinterval of $I_{\beta_{1,m-1}}^{(m)}$.
In particular provided that the quasi-embedding equation holds for some $y \in J'_{k-k'(m)-1}$, Lemma \ref{LKa} extends this quasi-embedding for points $x \in J'_{k-k'(m)-1}$ such that $x>y$. Lemma \ref{LKb} provides a similar extension for points in $J_{k-k'(m)}$.

\begin{lemma}\label{LKa}
	Given $n \geq 0$ and $0<m\leq n+1$ assume that $\gamma^{(n)}$ is a quasi-embedding of $f_{\lambda^{(m-1)},\pi^{(m-1)}}$ into $T^{(n,m-1)}$ for $x \in I^{(m-1)}\backslash f_{\lambda^{(m-1)},\pi^{(m-1)}}^{-1}\left(  I^{(n)}     \right) $, that for all $\alpha \in \mathcal{A}$ and $z \in \C$,
	\begin{equation}\label{C2'a}
	T_{\alpha}^{(n,m-1)} (z)={ \hat T}_{\alpha}^{(n,m-1)} (z).
	\end{equation}
	and that with ${\hat x}  = f_{\lambda^{(m-1)},\pi^{(m-1)}}^{-1}(x_{0,d}^{(m)})$  we have
	\begin{equation}\label{C5}
	T_{\beta_{1,m-1}}^{(n+1,m-1)}(z)=e^{i \theta_{\beta_{1,m-1}}^{(m-1)}}(z- \gamma^{(n+1)}({\hat x})) + \gamma^{(n+1)}(f_{\lambda^{(m-1)},\pi^{(m-1)}}({\hat x})).
	\end{equation}

	Furthermore for  $k(m)\leq k \leq k(m+1)$ assume that $\gamma^{(n+1)}$ is a quasi-embedding of $f_{\lambda^{(m-1)},\pi^{(m-1)}}$ into $T^{(n+1,m-1)}$ for $y \in I_{\beta_{1,m-1}}^{(m-1)} \cap J_{k-k'(m)-1}'$.
	
	If $y_k < x_{0,d}^{(m-1)}$, then $\gamma^{(n+1)}$ is a quasi-embedding of $f_{\lambda^{(m-1)},\pi^{(m-1)}}$ into $T^{(n+1,m-1)}$ for all $x \in [y, y_{k-k'(m)}]$. If $y_k \geq x_{0,d}^{(m-1)}$ then $\gamma^{(n+1)}$ is a quasi-embedding for $x \in [y, x_{0,d}^{(m-1)})$.
\end{lemma}

\begin{proof}
	As $x \in I_{\beta_{1,m-1}}^{(m-1)}$, by \eqref{C5'}, \eqref{C6'} and \eqref{C8'} we have $f_{\lambda^{(m-1)},\pi^{(m-1)}}(x) \in [ f_{\lambda^{(m-1)},\pi^{(m-1)}}(y), y_{k}   ]$, thus by \eqref{breq1}, \eqref{gamman} and continuity of $\gamma^{(n+1)}$ we get
	\begin{equation*}\label{KC3}
	\gamma^{(n+1)}(f_{\lambda^{(m-1)},\pi^{(m-1)}}(x)  )  = \gamma^{(n)}(f_{\lambda^{(m-1)},\pi^{(m-1)}}(x))+ \underline{\epsilon}_{k-1}.
	\end{equation*}
	Since $\gamma^{(n)}$ is a quasi-embedding of $f_{\lambda^{(m-1)},\pi^{(m-1)}}$ into $T^{(n,m-1)}$ for $ x \in [y, y_{k-k'(m)}]$ we have
	\begin{equation*}\label{KC2}
	\gamma^{(n)}(f_{\lambda^{(m-1)},\pi^{(m-1)}}(x)  )  = {T}^{(n,m-1)}(\gamma^{(n)}(x)).
	\end{equation*}
	Combining these two formulas and using \eqref{C2'a} we obtain
	\begin{equation*}
	\gamma^{(n+1)}(f_{\lambda^{(m-1)},\pi^{(m-1)}}(x)  )  = \left[ e^{i \theta_{\beta_{1,m-1}}^{(m-1)}}(\gamma^{(n)}(x)- \gamma^{(n)}(y)) + \gamma^{(n)}(f_{\lambda^{(m-1)},\pi^{(m-1)}}(y))     \right]+ \underline{\epsilon}_{k-1}.
	\end{equation*}
	Finally, using the definitions of breaking operator and breaking sequence one gets
	\begin{equation}\label{KC9}
	\gamma^{(n+1)}(f_{\lambda^{(m-1)},\pi^{(m-1)}}(x)  )  =  e^{i \theta_{\beta_{1,m-1}}^{(m-1)}}(\gamma^{(n+1)}(x)- \gamma^{(n+1)}(y)) + \gamma^{(n+1)}(f_{\lambda^{(m-1)},\pi^{(m-1)}}(y))   .
	\end{equation}

	Since $\gamma^{(n+1)}$ is a quasi-embedding of $f_{\lambda^{(m-1)},\pi^{(m-1)}}$ into $T^{(n+1,m-1)}$ for $y \in I_{\beta_{1,m-1}}^{(m-1)} \cap J_{k-k'(m)-1}'$ we have
	\begin{equation*}
	\gamma^{(n+1)}(f_{\lambda^{(m-1)},\pi^{(m-1)}}(y)  ) = T^{(n+1,m-1)}(\gamma^{(n+1)}(y)),
	\end{equation*}
	
	which combined with \eqref{C5}  gives that for any $z \in \C$,
	\begin{equation*}\label{KC10}
	T_{\beta_{1,m-1}}^{(n+1,m-1)}(z)=e^{i \theta_{\beta_{1,m-1}}^{(m-1)}}(z-\gamma^{(n+1)}(y)) + \gamma^{(n+1)}(f_{\lambda^{(m-1)},\pi^{(m-1)}}(y)).
	\end{equation*}
	Combined with \eqref{KC9}, we get 
	\begin{equation}\label{K2'a}
		\gamma^{(n+1)}(f_{\lambda^{(m-1)},\pi^{(m-1)}}(x)  ) = T^{(n+1,m-1)}(\gamma^{(n+1)}(x)).
	\end{equation}  
	for all $x \in [y, y_{k-k'(m)}]$ and therefore  $\gamma^{(n+1)}$ is a quasi-embedding of $f_{\lambda^{(m-1)},\pi^{(m-1)}}$ into $T^{(n+1,m-1)}$ in this interval.
	Moreover, it can be proved in a similar way that if $y_k \geq x_{0,d}^{(m-1)}$, then \eqref{K2'a} holds for all $x \in [y,  x_{0,d}^{(m-1)})$.
\end{proof}

\begin{lemma}\label{LKb}
	Given $n \geq 0$ and $0<m\leq n+1$ assume that $\gamma^{(n)}$ is a quasi-embedding of $f_{\lambda^{(m-1)},\pi^{(m-1)}}$ into $T^{(n,m-1)}$ for $x \in I^{(m-1)}\backslash f_{\lambda^{(m-1)},\pi^{(m-1)}}^{-1}\left(  I^{(n)}     \right) $, that for all $\alpha \in \mathcal{A}$, $z \in \C$ we have \eqref{C2'a}, and that with ${\hat x}  = f_{\lambda^{(m-1)},\pi^{(m-1)}}^{-1}(x_{0,d}^{(m)})$  we have \eqref{C5}.
	
	Furthermore for $k(m)\leq k \leq k(m+1)-1$ assume that $\gamma^{(n+1)}$ is a quasi-embedding of $f_{\lambda^{(m-1)},\pi^{(m-1)}}$ into $T^{(n+1,m-1)}$ for $y \in J_{k-k'(m)}$.
	
	If $y_k+\Delta \neq x_{0,d}^{(m-1)}$ then $\gamma^{(n+1)}$ is a quasi-embedding of $f_{\lambda^{(m-1)},\pi^{(m-1)}}$ into $T^{(n+1,m-1)}$ for $x \in [y, y_{k-k'(m)}+\Delta]$. If $y_k+\Delta = x_{0,d}^{(m-1)}$, then $\gamma^{(n+1)}$ is a quasi-embedding for $x \in [y, x_{0,d}^{(m-1)})$.
\end{lemma}

\begin{proof}
	
	Assume first that $y_k +\Delta \neq x_{0,d}^{(m-1)}$ and take $x \in [y, y_{k-k'(m)}+\Delta]$. As $(I^{(m-1)}\backslash I^{(m)})\cap J_k\neq \emptyset$, by Lemma \ref{LIET2} we must have $y_k +\Delta<x_{0,d}^{(m-1)}$ and thus $x \in I_{\beta_{1,m-1}}^{(m-1)}$. By \eqref{C7'} we have $f_{\lambda^{(m-1)},\pi^{(m-1)}}(x) \in [ f_{\lambda^{(m-1)},\pi^{(m-1)}}(y), y_k +\Delta   ]$, hence by \eqref{breq1}, \eqref{gamman} and continuity of $\gamma^{(n+1)}$ we get
	\begin{equation*}\label{Kb7}
	\gamma^{(n+1)}(f_{\lambda^{(m-1)},\pi^{(m-1)}}(x)  )  = \gamma^{(n)}(f_{\lambda^{(m-1)},\pi^{(m-1)}}(x)) e^{i \theta_{\beta_{1,n}}^{(n)}}+ \overline{\epsilon}_{k}.
	\end{equation*}
	As $[y, y_{k-k'(m)}+\Delta] \subseteq I^{(m-1)}\backslash f_{\lambda^{(m-1)},\pi^{(m-1)}}^{-1}\left(  I^{(n)}     \right) $, we have that $\gamma^{(n)}$ is a quasi-embedding of $f_{\lambda^{(m-1)},\pi^{(m-1)}}$ into $T^{(n,m-1)}$ for $x \in [y, y_{k-k'(m)}+\Delta]$ from whence we have
	\begin{equation*}
	\gamma^{(n)}(f_{\lambda^{(m-1)},\pi^{(m-1)}}(x)  )  = {T}^{(n,m-1)}(\gamma^{(n)}(x)).
	\end{equation*}
	Combining these two formulas and using \eqref{C2'a} we obtain
	\begin{equation*}\label{Kb8}
	\gamma^{(n+1)}(f_{\lambda^{(m-1)},\pi^{(m-1)}}(x)  )  = \left[ e^{i \theta_{\beta_{1,m-1}}^{(m-1)}}(\gamma^{(n)}(x)- \gamma^{(n)}(y)) + \gamma^{(n)}(f_{\lambda^{(m-1)},\pi^{(m-1)}}(y))     \right]e^{i \theta_{\beta_{1,n}}^{(n)}} + \overline{\epsilon}_{k}.
	\end{equation*}	
	As before, using the definitions of breaking operator and breaking sequence one gets \eqref{KC9}. We omit the conclusion of the proof as it is completely analogous to that of Lemma \ref{LKa}.
\end{proof}

\vspace*{0.2in}


We now proceed with the proof of Theorem \ref{fund_lemma}. The argument is structured as follows. The theorem holds trivially in the case $n\geq 0$ and from Lemma \ref{lfund_lemma1} in the case $n \geq 1$ and $m=n-1$. Next we assume, by induction on $m$, that given a fixed $n\geq1$, the theorem is true for $T^{(n,m)}$, with $0\leq m\leq n$ and also for  $T^{(n+1,m)}$, with $0< m\leq n+1$ and we prove it for  $T^{(n+1,m-1)}$.

We prove that $f_{\lambda^{(m-1)},\pi^{(m-1)}}$ is quasi-embedded into $T^{(n+1,m-1)}$ in  $I_{\beta_{1,m-1}}^{(m-1)}\backslash f_{\lambda^{(m-1)},\pi^{(m-1)}}^{-1}(I^{(n+1)})$ by induction in $k$, considering separate subintervals in $J^{(n+1)}$. In particular we achieve this by applying Lemmas \ref{LKa} and \ref{LKb} in an alternate way to extend the quasi-embedding throughout the interval. It follows that our theorem is true for $x \in I_{\beta_{1,m-1}}^{(m-1)}\backslash f_{\lambda^{(m-1)},\pi^{(m-1)}}^{-1}(I^{(n+1)})$.

To prove it is true for $I_{\alpha}^{(m-1)}\backslash f_{\lambda^{(m-1)},\pi^{(m-1)}}^{-1}(I^{(n+1)})$, with $\alpha \neq  \beta_{1,m-1}$, we separate cases $\varepsilon(m-1)=0$ and $\varepsilon(m-1)=1$. In both cases we distinguish between $\alpha \neq  \beta_{0,m-1}$, which requires the use of Lemmas \ref{LXi0} and \ref{LXi1}, and the case $\alpha =  \beta_{0,m-1}$ which follows from a distinct straightforward argument.

\vspace*{0.2in}

\begin{proof}[Proof of Theorem \ref{fund_lemma}]


Both statements in our theorem are trivial to prove for  $n \geq 0$ and $m=n$, as $I_{\alpha}^{(m)}\backslash f_{\lambda^{(m)},\pi^{(m)}}^{-1}\left(  I^{(n)}     \right)= \emptyset$. For $m=n-1$, both statements follow directly from Lemma \ref{lfund_lemma1}.


Given $n \geq 0$, we now assume the following.

\textit{(H1).} For all $0\leq m' \leq n$ and $\alpha \in \mathcal{A}$ that $\gamma^{(n)}$ is a quasi-embedding of $f_{\lambda^{(m')},\pi^{(m')}}$ into $T^{(n,m')}$ for $x \in I^{(m')}\backslash f_{\lambda^{(m')},\pi^{(m')}}^{-1}\left(  I^{(n)}     \right)$,
	and that for all $z \in \C$,
	\begin{equation*}\label{C2'}
	T_{\alpha}^{(n,m')} (z)={ \hat T}_{\alpha}^{(n,m')} (z).
	\end{equation*}
	
	\textit{(H2).} Given $0 < m \leq n+1$, we also assume that for all $\alpha \in \mathcal{A}$ that $\gamma^{(n+1)}$ is a quasi-embedding of $f_{\lambda^{(m)},\pi^{(m)}}$ into $T^{(n+1,m)}$ for $x \in I^{(m)}\backslash f_{\lambda^{(m)},\pi^{(m)}}^{-1}\left(  I^{(n+1)}     \right)$,	
	and that for $z \in \C$,
	\begin{equation}\label{C4'}
	T_{\alpha}^{(n+1,m)} (z)={ \hat T}_{\alpha}^{(n+1,m)} (z).
	\end{equation}	
%
	We need to relate the breaking sequence at the $(m-1)$-step of the Rauzy induction with our map $ T_{\alpha}^{(n+1,m-1)}$.
	
	\vspace{0.2in}
	
	\textbf{Case 1.} Fix $\alpha=\beta_{1,m-1}$. The Rauzy induction is either of type 1 or type 0 and we have $\beta_{\epsilon,m}= (\pi_{\epsilon}^{(m)})^{-1}(d)$.
	We prove now that $\gamma^{(n+1)}$ is a quasi-embedding of $f_{\lambda^{(m-1)},\pi^{(m-1)}}$ into $T^{(n+1,m-1)}$ for all $x \in I_{\alpha}^{(m-1)}\backslash  f_{\lambda^{(m-1)},\pi^{(m-1)}}^{-1}(I^{(n+1)})$ that is
	\begin{equation}\label{K2'}
	\gamma^{(n+1)}(f_{\lambda^{(m-1)},\pi^{(m-1)}}(x)  ) = T^{(n+1,m-1)}(\gamma^{(n+1)}(x)).
	\end{equation} 
	
	\textit{Step 1.} We begin by showing that we have \eqref{C5}, with ${\hat x}  = f_{\lambda^{(m-1)},\pi^{(m-1)}}^{-1}(x_{0,d}^{(m)})$.
	Assume first that $\varepsilon(m-1)=0$. From \eqref{Tnm0} and \eqref{C4'}, we have $$T_{\alpha}^{(n+1,m-1)}(z)=\left( {\hat T}_{  \beta_{0,m-1}   }^{(n+1,m)} \right)^{-1} \circ {\hat T}_{\alpha}^{(n+1,m)}(z),$$ for all $z \in \C$. By definition of $T_{\alpha}^{(n+1,m-1)}$ and Lemma \ref{LA} we get \eqref{C5}.

	Assume now that $\varepsilon(m-1)=1$. In this case, we have $ f_{\lambda^{(m-1)},\pi^{(m-1)}}(x) = f_{\lambda^{(m)},\pi^{(m)}}(x)   $ for $x \in I_{\alpha}^{(m)}$. In particular, if $x \in  I_{\alpha}^{(m)}\backslash   f_{\lambda^{(m)},\pi^{(m)}}^{-1}(I^{(n+1)})  $, then $x \in  I_{\alpha}^{(m)}\backslash   f_{\lambda^{(m-1)},\pi^{(m-1)}}^{-1}(I^{(n+1)})  $ as well.

	 By (H2) and \eqref{A7} we get
	\begin{equation}\label{C*1}
	\xi_{\pi_1^{(m)}(\alpha)}^{n+1,m}- e^{i \theta_{\alpha}^{(m)}}\gamma_{0,\pi_{0}^{(m)}(\alpha)}^{n+1,m}=\gamma^{(n+1)}\left(  f_{\lambda^{(m-1)},\pi^{(m-1)}}(x)  \right) - e^{i \theta_{\alpha}^{(m-1)}} \gamma^{(n+1)}(x),
	\end{equation} 
	for all $x \in  I_{\alpha}^{(m)}\backslash   f_{\lambda^{(m-1)},\pi^{(m-1)}}^{-1}(I^{(n+1)})$. By \eqref{Tnm1} and \eqref{C4'} we have $T_{\alpha}^{(n+1,m-1)}={\hat T}_{\alpha}^{(n+1,m)}$, hence by \eqref{C*1} and \eqref{hatTnm} we get \eqref{K2'} for $x \in   I_{\alpha}^{(m)}\backslash   f_{\lambda^{(m-1)},\pi^{(m-1)}}^{-1}(I^{(n+1)})  $.

	Since $\gamma^{(n+1)}$ is a continuous map and $f_{\lambda^{(m-1)},\pi^{(m-1)}}$ is continuous at ${\hat x}$,  we get \eqref{K2'} for $x={\hat x}$. Since $T_{\alpha}^{(n+1,m-1)}={\hat T}_{\alpha}^{(n+1,m)}$,  \eqref{C5} holds as well.
	
	\vspace{0.2in}
	
	
	 \textit{Step 2.} Recall we denote by $J^{(n+1)}$ the ordering of $\left\{   f_{\lambda,\pi}^k\left(I^{(n)}\backslash I^{(n+1)}\right)   \right\}_{0\leq k<r(n)}$ and that we have the relations \eqref{C5'}-\eqref{C8'}.

	By Lemma \ref{LIET2},  $ J_{k(m)-1} \subseteq I^{(m)}$ and $J_{k(m)}  \subseteq I^{(m-1)}\backslash I^{(m)}$. Thus, either $y_{k(m)-1}+\Delta \leq x_{0,d}^{(m)}<y_{k(m)}$ or $x_{0,d}^{(m)}=y_{k(m)}$.

	Assuming first that $y_{k(m)-1}+\Delta \leq x_{0,d}^{(m)}<y_{k(m)}$, from \eqref{C5} we get that $\gamma^{(n+1)}$ is a quasi-embedding of $f_{\lambda^{(m-1)},\pi^{(m-1)}}$ into $T^{n+1,m-1}$ for $y=f_{\lambda^{(m-1)},\pi^{(m-1)}}^{-1}(x_{0,d}^{(m)})$, that is
	\begin{equation}\label{K1'}
	\gamma^{(n+1)}(f_{\lambda^{(m-1)},\pi^{(m-1)}}(y)  ) = T^{(n+1,m-1)}(\gamma^{(n+1)}(y)).
	\end{equation}
	Since we are assuming (H1) we can apply Lemma \ref{LKa}, and thus we have \eqref{K2'} either for all $x \in  I_{\alpha}^{(m-1)}$ if $y_{k(m)}= x_{0,d}^{(m-1)}$, or for all $x \in [ f_{\lambda^{(m-1)},\pi^{(m-1)}}^{-1}(x_{0,d}^{(m)})  , y_{k(m) - k'(m)} ]$ if $y_{k(m)}< x_{0,d}^{(m-1)}$. In particular we have \eqref{K1'} with $y =y_{k(m) - k'(m)}$.
	
	Now assume that $x_{0,d}^{(m)}=y_{k(m)}$. By \eqref{C5} we also have \eqref{K1'} with $y=y_{k(m)- k'(m)}$. Therefore by Lemma \ref{LKa} we have \eqref{K2'} either for all $x \in  I_{\alpha}^{(m-1)}$ if $y_{k(m)}+\Delta = x_{0,d}^{(m-1)}$, or for all $ x \in [ f_{\lambda^{(m-1)},\pi^{(m-1)}}^{-1}(x_{0,d}^{(m)})  , y_{k(m) - k'(m)}+\Delta ]$ if $y_{k(m)}+\Delta < x_{0,d}^{(m-1)}$.
	
	\vspace{0.2in}
	
	\textit{Step 3.} Now assume, by induction on $k$, for $k(m)+1 \leq k \leq k(m+1)$, and with $y_{k-1} +\Delta< x_{0,d}^{(m-1)}$, that $\gamma^{(n+1)}$ is a quasi-embedding of $f_{\lambda^{(m-1)},\pi^{(m-1)}}$ into $T^{n+1,m-1}$ for all $ x \in [ f_{\lambda^{(m-1)},\pi^{(m-1)}}^{-1}(x_{0,d}^{(m)})  , y_{k - k'(m)-1}+\Delta ]$. In particular we have \eqref{K1'} with $y=  y_{k-k'(m)-1} +\Delta $. 
	Thus by Lemma \ref{LKa} we have \eqref{K2'} either for all $x \in  I_{\alpha}^{(m-1)}$ if $y_{k}\geq x_{0,d}^{(m-1)}$, or for all $x \in [ f_{\lambda^{(m-1)},\pi^{(m-1)}}^{-1}(x_{0,d}^{(m)})  , y_{k - k'(m)} ]$ if $y_{k}< x_{0,d}^{(m-1)}$.
	In particular we get that  $\gamma^{(n+1)}$ is a quasi-embedding of $f_{\lambda^{(m-1)},\pi^{(m-1)}}$ into $T^{n+1,m-1}$ for $y=y_{k-k'(m)}$. Since we are assuming (H1) we can apply Lemma \ref{LKb} and thus we have \eqref{K2'} either for all $x \in  I_{\alpha}^{(m-1)}$ if $y_{k}+\Delta = x_{0,d}^{(m-1)}$, or for all $x \in [ f_{\lambda^{(m-1)},\pi^{(m-1)}}^{-1}(x_{0,d}^{(m)})  , y_{k - k'(m)} +\Delta]$ if $y_{k}\neq x_{0,d}^{(m-1)}$.
	Since $f_{\lambda^{(m-1)},\pi^{(m-1)}}^{-1}([x_{0,d}^{(m)},  x_{0,d}^{(m-1)}  )   )  = I_{\alpha}^{(m-1)} \cap f_{\lambda^{(m-1)},\pi^{(m-1)}}^{-1}(   I_{\beta_{0,m-1}}^{(m-1)}  )  $, this shows that we have \eqref{K2'} for all $x \in I_{\alpha}^{(m-1)} \cap f_{\lambda^{(m-1)},\pi^{(m-1)}}^{-1}(   I_{\beta_{0,m-1}}^{(m-1)}  ) $.
	In particular if $\varepsilon(m-1)=0$, this shows that $\gamma^{(n+1)}$ is a quasi-embedding of $f_{\lambda^{(m-1)},\pi^{(m-1)}}$ into $T^{n+1,m-1}$ for all $x \in I_{\alpha}^{(m-1)}$. If $\varepsilon(m-1)=1$, since $  f_{\lambda^{(m-1)},\pi^{(m-1)}}^{-1}(   I_{\beta_{0,m-1}}^{(m-1)}  )= I_{\alpha}^{(m-1)} \backslash I_{\alpha}^{(m)}  $ and we already proved that \eqref{K2'} holds for all $x \in I_{\alpha}^{(m)} \backslash f_{\lambda^{(m-1)},\pi^{(m-1)}}^{-1}(  I^{(n+1)}  ) $, this shows that it is true for all $x \in I_{\alpha}^{(m-1)} \backslash f_{\lambda^{(m-1)},\pi^{(m-1)}}^{-1}(  I^{(n+1)}  ) $.	
	
	\vspace{0.2in}	
	
	\textit{Step 4.} Combining \eqref{K2'} and \eqref{C5}, for any $x \in I_{\alpha}^{(m-1)} \backslash   f_{\lambda^{(m-1)},\pi^{(m-1)}}^{-1}(  I^{(n+1)} )   $ and $z \in \C $ 
	replacing $x= x_{0,\hpi^{(m-1)(d)}}^{(m-1)} - \delta$ and taking $\delta\rightarrow 0^+$, we get
	\begin{equation}\label{C*3}
	T_{\alpha}^{(n+1,m-1)}(z)= e^{i \theta_{\alpha}^{(m-1)}}\left(  z - \gamma_{0,\hpi^{(m-1)}(d)}^{n+1,m-1}    \right)  + \gamma_{0,d}^{n+1,m-1},
	\end{equation}
	and this can be written as
	\begin{equation*}\label{C*3a}
	T_{\beta_{1,m-1}}^{(n+1,m-1)}(z) = {\hat T}_{\beta_{1,m-1}}^{(n+1,m-1)}(z).
	\end{equation*}
	
	\vspace{0.2in}
	
	In the next cases we establish a relation between $T_{\alpha}^{(n+1,m-1)}$, when  $\alpha\neq \beta_{1,m-1}$ and the breaking sequence at the step $n+1$.
	
	 Note first that since we are assuming that $\gamma^{(n+1)}$ is a quasi-embedding of $f_{\lambda^{(m)},\pi^{(m)}}$ into $T^{(n+1,m)}$ for $x \in I_{\alpha}^{(m)}\backslash f_{\lambda^{(m)},\pi^{(m)}}^{-1}\left(  I^{(n+1)}     \right) $ it follows that for these values of $x$ we have
	\begin{equation}\label{CX1''}
	T^{(n+1,m)}(\gamma^{(n+1)}(x))= \gamma^{(n+1)}(f_{\lambda^{(m)},\pi^{(m)}}(x)).
	\end{equation}

	\textbf{Case 2.} Set $\alpha\neq \beta_{1,m-1}$ and $\varepsilon(m-1)=0$.
	Since $f_{\lambda^{(m-1)}, \pi^{(m-1)} }^{-1}( x_{0,d}^{(m)}  )  = x_{0,\hpi^{(m-1)}(d)-1}^{(m-1)}   $, by \eqref{C5} we get
	\begin{equation*}\label{}
	T_{\beta_{1,m-1}}^{(n+1,m-1)}(z)= e^{i \theta_{\beta_{1,m-1}}^{(m-1)}}\left(  z - \gamma_{0,\hpi^{(m-1)}(d)-1}^{n+1,m-1}    \right)  + \gamma_{0,d-1}^{n+1,m-1},
	\end{equation*}
	which by \eqref{C*3} and \eqref{xis} shows that $\xi_{d-1}^{n+1,m-1} =  \gamma_{1,d-1}^{n+1,m-1}  $. Hence by Lemma \ref{LXi0} we get that $ T_{\alpha}^{(n+1,m-1)} = {\hat T}_{\alpha}^{(n+1,m-1)} $.
	
	By \eqref{Tnm0} and \eqref{C4'} we get that $ T_{\alpha}^{(n+1,m-1)} = T_{\alpha}^{(n+1,m)} $ and by \eqref{CX1''} we get
	$$  T_{\alpha}^{(n+1,m-1)}(\gamma^{(n+1)}(x))= \gamma^{(n+1)}(f_{\lambda^{(m)},\pi^{(m)}}(x)),   $$
    for $x \in I_{\alpha}^{(m)} \backslash  f_{\lambda^{(m)},\pi^{(m)}}^{-1}( I^{(n+1)} ) $. Since $f_{\lambda^{(m-1)},\pi^{(m-1)}}(x)  =  f_{\lambda^{(m)},\pi^{(m)}}(x) $, for $x \in  I_{\alpha}^{(m)}   $, we get
    \begin{equation}\label{C*4}
     T_{\alpha}^{(n+1,m-1)}(\gamma^{(n+1)}(x))= \gamma^{(n+1)}(f_{\lambda^{(m-1)},\pi^{(m-1)}}(x)),
    \end{equation}
	for all $x \in I_{\alpha}^{(m)} \backslash  f_{\lambda^{(m-1)},\pi^{(m-1)}}^{-1}( I^{(n+1)} ) $. In particular, for $\alpha \neq \beta_{0,m-1}  $ we get \eqref{C*4} for $x \in I_{\alpha}^{(m-1)} \backslash  f_{\lambda^{(m-1)},\pi^{(m-1)}}^{-1}( I^{(n+1)} ) $.
	
	Now take $\alpha = \beta_{0,m-1}$ and $x \in (I_{\alpha}^{(m-1)} \backslash I_{\alpha}^{(m)}) \backslash  f_{\lambda^{(m-1)},\pi^{(m-1)}}^{-1}( I^{(n+1)} )$. Since $f_{\lambda^{(m-1)},\pi^{(m-1)}}^{-1}(x) \in I_{\beta_{1,m-1}}^{(m-1)}$, we get by \eqref{K2'},
	$$\gamma^{n+1}(x)= T_{\beta_{1,m-1}}^{(n+1,m-1)}(\gamma^{(n+1)}(f_{\lambda^{(m-1)},\pi^{(m-1)}}^{-1}(x)))$$
	and since $x \in I_{\alpha}^{(m-1)}$, by \eqref{Tnm0} this gives
	$$  T_{\alpha}^{(n+1,m-1)}(\gamma^{(n+1)}(x))= T_{\beta_{1,m-1}}^{(n+1,m)}(\gamma^{(n+1)}(f_{\lambda^{(m-1)},\pi^{(m-1)}}^{-1}(x))). $$
	As $I_{\beta_{1,m-1}}^{(m)} =  I_{\beta_{1,m-1}}^{(m-1)} $ and $f_{\lambda^{(m-1)},\pi^{(m-1)}}^2(x')=f_{\lambda^{(m)},\pi^{(m)}}(x')$, for $x' \in I_{\beta_{1,m-1}}$, we get that $f_{\lambda^{(m-1)},\pi^{(m-1)}}^{-1}(x) \in I_{\beta_{1,m-1}}^{(m-1)} \backslash  f_{\lambda^{(m)},\pi^{(m)}}^{-1}( I^{(n+1)} ) $ and $f_{\lambda^{(m-1)},\pi^{(m-1)}}^{-1}(x)  =  f_{\lambda^{(m)},\pi^{(m)}}^{-1} \circ f_{\lambda^{(m-1)},\pi^{(m-1)}} (x)   $, thus by \eqref{CX1''} we get \eqref{C*4} for $x \in I_{\alpha}^{(m-1)} \backslash  f_{\lambda^{(m-1)},\pi^{(m-1)}}^{-1}( I^{(n+1)} ) $.\\
	
	\textbf{Case 3.} Now assume $\varepsilon(m-1)=1$ and $\alpha \neq \beta_{1,m-1}$. By \eqref{Tnm1} and \eqref{C4'} we have $ T_{\beta_{1,m-1}}^{(n+1,m-1)} = {\hat T}_{\beta_{1,m-1}}^{(n+1,m)} $, and combining this with \eqref{A7} and \eqref{hatTnm} we get
	$$  T_{\beta_{1,m-1}}^{(n+1,m-1)}(z) =  e^{i \theta_{\beta_{1,m-1}}^{(m-1)}} (z  -  \gamma_{0,\hpi^{(m)}(d)}^{n+1,m}) + \xi_{d}^{n+1,m} , $$
	for any $z \in \C$. As $  \gamma_{1,d}^{n+1,m-1} =  \xi_{d}^{n+1,m-1} $, from \eqref{C*3} we get
	$$   \xi_{d}^{n+1,m-1}  - e^{i \theta_{\beta_{1,m-1}}^{(m-1)}} \gamma_{0,\hpi^{(m-1)}(d)}^{n+1,m-1}   = \xi_{d}^{n+1,m}  - e^{i \theta_{\beta_{1,m-1}}^{(m-1)}} \gamma_{0,\hpi^{(m-1)}(d)}^{n+1,m} ,  $$
	which by \eqref{xis} with $j=d-1$, gives 
	$$   \xi_{d-1}^{n+1,m-1} =   e^{i \theta_{\beta_{1,m-1}}^{(m-1)}} \left(    \gamma_{0,\hpi^{(m-1)}(d)-1}^{n+1,m-1}  -\gamma_{0,\hpi^{(m-1)}(d)}^{n+1,m}   \right)  +   \xi_{d}^{n+1,m}.  $$
	Recalling \eqref{xi11} we have $  \gamma_{0,\hpi^{(m-1)}(d)-1}^{n+1,m-1}  = \gamma_{0,\hpi^{(m)}(d)-1}^{n+1,m}    $ and again by \eqref{xis} we get that $ \xi_{d-1}^{n+1,m-1} = \xi_{d-1}^{n+1,m} $. Thus, by Lemma \ref{LXi1}, \eqref{hatTnm},\eqref{Tnn} and \eqref{Tnm1} we obtain $ T_{\alpha}^{(n+1,m-1)} = {\hat T}_{\alpha}^{(n+1,m-1)} $.

	By a reasoning analogous to the case $\varepsilon(m-1) = 0$, we have that \eqref{C*4} is true for all $x \in I_{\alpha}^{(m)} \backslash f_{\lambda^{(m-1)},\pi^{(m-1)}}^{-1} (I^{(n+1)}) $. In particular, for $\alpha \neq \beta_{0,m-1}$ we get \eqref{C*4} for all $x \in I_{\alpha}^{(m-1)} \backslash f_{\lambda^{(m-1)},\pi^{(m-1)}}^{-1} (I^{(n+1)}) $.
	
	Now take $\alpha = \beta_{0,m-1}$ and $x \in I_{\alpha}^{(m-1)}  \backslash  f_{\lambda^{(m-1)},\pi^{(m-1)}}^{-1}( I^{(n+1)} )$. Since $f_{\lambda^{(m-1)},\pi^{(m-1)}}^{-1}(x) \in I_{\beta_{1,m-1}}^{(m-1)}$, we get by \eqref{K2'} and \eqref{Tnm1},
	$$\gamma^{n+1}(x)= T_{\beta_{1,m-1}}^{(n+1,m-1)}(\gamma^{(n+1)}(f_{\lambda^{(m-1)},\pi^{(m-1)}}^{-1}(x)))$$
	and since $x \in I_{\alpha}^{(m-1)}$, by \eqref{Tnm1} this gives
	$$  T_{\alpha}^{(n+1,m-1)}(\gamma^{(n+1)}(x))= T_{\alpha}^{(n+1,m)}(\gamma^{(n+1)}(f_{\lambda^{(m-1)},\pi^{(m-1)}}^{-1}(x))). $$
	As $f_{\lambda^{(m-1)},\pi^{(m-1)}}^{-1}(I_{\alpha}^{(m-1)}) =  I_{\alpha}^{(m-1)} $ and $f_{\lambda^{(m-1)},\pi^{(m-1)}}^2(x')=f_{\lambda^{(m)},\pi^{(m)}}(x')$, for $x' \in I_{\alpha}$, we get that $f_{\lambda^{(m-1)},\pi^{(m-1)}}^{-1}(x) \in I_{\alpha}^{(m-1)} \backslash  f_{\lambda^{(m)},\pi^{(m)}}^{-1}( I^{(n+1)} ) $ and that $f_{\lambda^{(m-1)},\pi^{(m-1)}}^{-1}(x)  = f_{\lambda^{(m)},\pi^{(m)}}^{-1} \circ f_{\lambda^{(m-1)},\pi^{(m-1)}} (x)$, thus by \eqref{CX1''} we get \eqref{C*4} for $x \in I_{\alpha}^{(m-1)} \backslash  f_{\lambda^{(m-1)},\pi^{(m-1)}}^{-1}( I^{(n+1)} ) $.
	
	\vspace*{0.2in}
	
	\textbf{Conclusion.} We proved that for all $z \in \C$,
	\begin{equation*}\label{}
	T_{\alpha}^{(n+1,m-1)} (z)={ \hat T}_{\alpha}^{(n+1,m-1)} (z),
	\end{equation*}
	and from \eqref{C*4} we get for all $\alpha \in \mathcal{A}$ that $\gamma^{(n+1)}$ is a quasi-embedding of $f_{\lambda^{(m-1)},\pi^{(m-1)}}$ into $T^{(n+1,m-1)}$ for $x \in I^{(m-1)}\backslash f_{\lambda^{(m-1)},\pi^{(m-1)}}^{-1}\left(  I^{(n+1)}     \right)$.
	Thus for all $0\leq m \leq n+1$ and $\alpha \in \mathcal{A}$ we have that \eqref{C4'} and \eqref{CX1''} hold and therefore $\gamma^{(n+1)}$ is a quasi-embedding of $f_{\lambda^{(m)},\pi^{(m)}}$ into $T^{(n+1,m)}$ for $x \in I^{(m)}\backslash f_{\lambda^{(m)},\pi^{(m)}}^{-1}\left(  I^{(n+1)}     \right)$.
	
	This shows that for all $n\geq 0$, $0\leq m \leq n$ and $\alpha \in \mathcal{A}$ that $\gamma^{(n)}$ is a quasi-embedding of $f_{\lambda^{(m)},\pi^{(m)}}$ into $T^{(n,m)}$ for $x \in I^{(m)}\backslash f_{\lambda^{(m)},\pi^{(m)}}^{-1}\left(  I^{(n)}     \right)$ and for all $z \in \C$ we have \eqref{X2'}. This finishes our proof.
	
%
	
\end{proof}

\section{Existence of embeddings of IETs into PWIs}

In this section we prove the existence of non-trivial embeddings of IETs into PWIs. We introduce the family $\mathcal{F}_{\theta}$ of PWIs which are $\theta$-adapted to an IET $(\lambda,\pi)$ and show, in Theorem \ref{tOmembed}, that when $\gamma_{\theta}^{(n)}$ converges to a topological embedding $\gamma_{\theta}$, then the latter is an isometric embedding of $(I,f_{\lambda,\pi})$ into any $\theta$-adapted PWI.
We recall some classical notions of the theory of IETs, in particular the Zorich cocycle and the characterization of its Oseledets flags and associated Lyapunov spectrum, as well as the translation surface of genus $g(\mathfrak{R})$ associated to an IET.  We use these results to prove an important bound on a quantitity determined by the Rauzy cocycle in Lemma \ref{lsumrauzy}.
We introduce a submanifold $W_{[\lambda],\pi}^{\delta}$ of the torus $\T^{\mathcal{A}}$ related to the Oseledets flags of the Zorich cocycle for the underlying IET and use Lemma \ref{lsumrauzy} to determine a bound for the sequence $\theta^{(n)}$ when $\theta \in W_{[\lambda],\pi}^{\delta}$. This result together with Theorem \ref{tOmembed}  are the key ingredients in the proof of Theorem \ref{taexistembeds} which states that for a full measure set of IETs, if $\theta \in W_{[\lambda],\pi}^{\delta}$, then $\gamma_{\theta}^{(n)}$ converges to a Lipschitz map $\gamma_{\theta}$, which is an isometric embedding of $(I,f_{\lambda,\pi})$ into any $\theta$-adapted PWI.
The embedding resulting from Theorem \ref{taexistembeds} may, however, be trivial. Thus we define a submanifold $\mathcal{W}_{[\lambda],\pi}^{\delta} \subset W_{[\lambda],\pi}^{\delta}$ which we show, in Theorem \ref{taeexistnontrivembeds}, has full measure when $g(\mathfrak{R})\geq 2$, for which the embedding $\gamma_{\theta}$ is guaranteed to be non-trivial.

Given $(\lambda,\pi)\in \R_+^{\mathcal{A}}\times \mathfrak{S}({\mathcal{A}})$, recall we denote by $\Theta_{\lambda,\pi}$ the set of all $\theta \in \T^{\mathcal{A}}$ such that for all $n \geq 0$, $\gamma_{\theta}^{(n)}:I \rightarrow \C$ is an injective map. Let $\Theta_{\lambda,\pi}'$ denote the set of all $\theta \in \Theta_{\lambda,\pi}$ for which there exists a topological embedding $\gamma_{\theta}:I \rightarrow \C$ such that for all $x \in I$,
\begin{equation*}\label{gamma}
\gamma_{\theta}(x)= \lim_{n \rightarrow +\infty}\gamma_{\theta}^{(n)}(x).
\end{equation*}

Furthermore, given $\theta \in \Theta_{\lambda,\pi}'$, we say that a PWI $T:X \rightarrow X$ together with a partition $\{X_{\alpha}\}_{\alpha \in \mathcal{A}}$ is $\theta$-\emph{adapted to} $(\lambda,\pi)$ if for all $\alpha \in \mathcal{A}$,\vspace*{0.2cm}\\
\vspace*{0.2cm}
i) $X_{\alpha}\supseteq \gamma_{\theta}(I_{\alpha})$ ;\\
ii) with $x_j= x_{0,j}^{(0)}$, and
\begin{equation}\label{PWIinv}
T_{\alpha}(z)=e^{i \theta_{\alpha}}\left( z -  \gamma_{\theta}(x_{\pi_0(\alpha)-1})   \right) +   \gamma_{\theta}\left(f_{\lambda,\pi}(x_{\pi_0(\alpha)-1})\right),
\end{equation}
for all $z \in \C$, we have $T(z)=T_{\alpha}(z)$, for all $z \in X_{\alpha}$.

We denote the family of PWIs which are $\theta$-\emph{adapted to} $(\lambda,\pi)$ by $\mathcal{F}_{\theta}$.\\	

Recall that we say there is a \emph{embedding} of an IET $(I,f_{\lambda,\pi})$ into a PWI $(X,T)$ if there exists a topological embedding $\gamma:I \rightarrow \C$ such that for all $x\in I$,
\begin{equation*}\label{embed}
\gamma\circ f_{\lambda,\pi}(x) = T\circ \gamma (x).
\end{equation*}

Given $x \in I$, consider the family $\Pi(x)$ of points $0=t_0<t_1<...<t_N=x$. Given $\theta \in \Theta_{\lambda,\pi}'$ define a map $\mathcal{L}_{\theta}: I \rightarrow \R_+   $ by
$$ \mathcal{L}_{\theta}(x)= \sup_{(t_0,...,t_N) \in \Pi(x)} \sum_{j=0}^{N-1}\left|  \gamma_{\theta}(t_{j+1}) - \gamma_{\theta}(t_{j})    \right|.  $$
We say a map $\gamma_{\theta}$ is an \textit{isometric embedding} of an IET $(I,f_{\lambda,\pi})$ into a PWI $(X,T)$ if it is an embedding and $\mathcal{L}_{\theta}(x)=x$ for all $x \in I$.

The following theorem states that when $\gamma_{\theta}^{(n)}$ converges to a topological embedding $\gamma_{\theta}$ it is also an isometric embedding of $(I,f_{\lambda,\pi})$ into any PWI which is $\theta$-adapted to $(\lambda,\pi)$. The proof follows from estimates related to the facts that the restriction of any PWI in $\mathcal{F}_{\theta}$ to $\gamma_{\theta}(I)$ can be approximated by the map $T^{(n,0)}$ with increasing precision as $n \rightarrow +\infty$, and that Theorem \ref{fund_lemma} guarantees that $\gamma^{n}$ is a quasi-embedding of $f_{\lambda,\pi}$ into $T^{(n,0)}$ for points in $I \backslash f_{\lambda,\pi}^{-1}(I^{(n)})$ which implies that the conjugacy between these two maps only fails to hold for points in an interval which is decreasing with $n$.
\begin{theorem}\label{tOmembed}
	Let $(\lambda,\pi)\in \R_+^{\mathcal{A}}\times \mathfrak{S}({\mathcal{A}})$, $\theta \in \Theta_{\lambda,\pi}'$ and $(X,T)$ be a PWI $\theta$-adapted to $(\lambda,\pi)$. Then $\gamma_{\theta}$ is an isometric embedding of $(I,f_{\lambda,\pi})$ into $(X,T)$.
\end{theorem}

\begin{proof}
	For any map $g:I \rightarrow \C$ denote $\| g  \|_{\infty} = \sup_{x \in I}|g(x)|$.
	
	As $f_{\lambda,\pi}$ is a bijective map we have
	$$ \| \gamma_{\theta} \circ f_{\lambda,\pi} -  \gamma_{\theta}^{(n)} \circ f_{\lambda,\pi}   \|_{\infty} = \| \gamma_{\theta}  -  \gamma_{\theta}^{(n)}   \|_{\infty}, $$
	which as $\theta \in \Theta_{\lambda,\pi}'$, shows that
	\begin{equation}\label{tA1}
	\lim_{n \rightarrow +\infty}\| \gamma_{\theta} \circ f_{\lambda,\pi} -  \gamma_{\theta}^{(n)} \circ f_{\lambda,\pi}   \|_{\infty}=0.
	\end{equation}
	
	From \eqref{hatTnm} and Theorem \ref{fund_lemma}, for any $\alpha \in \mathcal{A}$ and $x \in I$ we have
	$$  T_{\alpha}^{(n,0)}(\gamma_{\theta}^{(n)}(x)) = e^{i \theta_{\alpha}}\left( \gamma_{\theta}^{(n)}(x) -  \gamma_{0,\pi_0(\alpha)-1}^{n,0}  \right) +   \gamma_{1,\pi_1(\alpha)-1}^{n,0},   $$
	and by \eqref{PWIinv} applying the triangle inequality we get
	\begin{equation*}
	\begin{array}{lr}
	  \|   T_{\alpha}^{(n,0)} \circ \gamma_{\theta}^{(n)}   - T_{\alpha} \circ \gamma_{\theta}    \|_{\infty} \leq 
	\\
	 \|     \gamma_{\theta}^{(n)}   -  \gamma_{\theta}     \|_{\infty}  +    |  \gamma_{\theta}(x_{\pi_0(\alpha)-1})   - \gamma_{0,\pi_0(\alpha)-1}^{n,0}|   +   |\gamma_{1,\pi_1(\alpha)-1}^{n,0}   -  \gamma_{\theta}\left(f_{\lambda,\pi}(x_{\pi_0(\alpha)-1})\right) |,  
	 \end{array} 
	\end{equation*}
		
	which, as $\theta \in \Theta_{\lambda,\pi}'$, shows that
	\begin{equation}\label{tA2}
	\lim_{n \rightarrow +\infty} \|   T^{(n,0)} \circ \gamma_{\theta}^{(n)}   - T \circ \gamma_{\theta}    \|_{\infty}=0.
	\end{equation}
	
	By Theorem \ref{fund_lemma}, $\gamma^{(n)}$ is a quasi-embedding of $f_{\lambda,\pi}$ into $T^{(n,0)}$ for all $x \in I \backslash f_{\lambda,\pi}^{-1}(I^{(n)})$ and thus we have
	$$  \gamma^{(n)}\left(   f_{\lambda,\pi}(x)   \right) = T^{(n,0)}\left(  \gamma^{(n)}(x) \right),$$
	in particular this gives
	$$ \| \gamma^{(n)} \circ  f_{\lambda,\pi}   -  T^{(n,0)} \circ \gamma^{(n)} \|_{\infty} \leq \sup_{x \in f_{\lambda,\pi}^{-1}(I^{(n)}) } \left|    \gamma^{(n)}\left(   f_{\lambda,\pi}(x)   \right) - T^{(n,0)}\left(  \gamma^{(n)}(x)  \right)     \right|.$$
	For a sufficiently large $N>0$ we have $f_{\lambda,\pi}^{-1}(I^{(n)})   \subseteq I_{\pi_1^{-1}(1)}   $, whenever $n>N$.
	
	~
	
	As $T^{(n,0)}(\gamma_{\theta}^{(n)}(x_{\hpi^{(n)}(1)}^{(n)}))= x_1^{(n)} \in \overline{I^{(n)}} $ and since $T_{\pi_1^{-1}(1)}^{(n,0)}$ is an isometry, we get that
	$$   \sup_{x \in f_{\lambda,\pi}^{-1}(I^{(n)}) } \left|   T^{(n,0)}\left(  \gamma^{(n)}(x)  \right)    \right|  \leq 2 |I^{(n)}|.$$
	Since $\sup_{x \in f_{\lambda,\pi}^{-1}(I^{(n)}) } \left|  \gamma^{(n)}\left(   f_{\lambda,\pi}(x)   \right)    \right|  \leq  |I^{(n)}|$ and $|I^{(n)}| \rightarrow 0$ as $n \rightarrow + \infty$, this shows that
	\begin{equation}\label{tA3}
	\lim_{n \rightarrow +\infty} \| \gamma^{(n)} \circ  f_{\lambda,\pi}   -  T^{(n,0)} \circ \gamma^{(n)}    \|_{\infty}=0.
	\end{equation}
	
	By the triangle inequality we have
	$$ \| \gamma_{\theta} \circ  f_{\lambda,\pi}   -  T \circ \gamma_{\theta}    \|_{\infty} \leq \| \gamma_{\theta} \circ f_{\lambda,\pi} -  \gamma_{\theta}^{(n)} \circ f_{\lambda,\pi}   \|_{\infty}   +  \| \gamma^{(n)} \circ  f_{\lambda,\pi}   -  T^{(n,0)} \circ \gamma^{(n)}    \|_{\infty}  + \|   T^{(n,0)} \circ \gamma_{\theta}^{(n)}   - T \circ \gamma_{\theta}    \|_{\infty}   . $$

	Taking the limit as $n \rightarrow  + \infty$ and by \eqref{tA1}, \eqref{tA2} and \eqref{tA3} we get
	\begin{equation*}\label{}
	\gamma_{\theta}\circ f_{\lambda,\pi}(x) = T \circ \gamma_{\theta} (x),
	\end{equation*}
	for all $x \in I$, which proves that $\gamma_{\theta}$ is an embedding of $(I,f_{\lambda,\pi})$ into $(X,T)$.
	
	Finally, given $x \in I$, consider $0=t_0<t_1<...<t_N=x$. For all $n\geq 0$, $\gamma_{\theta}^{(n)} \in \mathcal{PL}(|I|)$ from which follows that $\left| \gamma_{\theta}^{(n)}(t_{j+1}) - \gamma_{\theta}^{(n)}(t_{j})   \right|= \left|  t_{j+1}-t_{j}  \right|$,  for any $j=0,...,N-1$. Hence, as $\theta \in \Theta_{\lambda,\pi}'$, we get
	$$ x = \sum_{j=0}^{N-1}\left|  \gamma_{\theta}^{(n)}(t_{j+1}) - \gamma_{\theta}^{(n)}(t_{j})    \right| \rightarrow  \sum_{j=0}^{N-1}\left|  \gamma_{\theta}(t_{j+1}) - \gamma_{\theta}(t_{j})    \right|, \ \textrm{as} \ n \rightarrow + \infty, $$
	which shows that $\mathcal{L}_{\theta}(x)=x$ finishing our proof.
\end{proof}


Following \cite{AF,AL}, let $\mathbb{P}_+^{\mathcal{A}} = \mathbb{P}(\R_+^{\mathcal{A}}) \simeq \mathbb{P}_+^{\mathcal{A}}$ denote the projectivization of $\R_{+}^{\mathcal{A}}$. Let $\mathfrak{R} \subseteq \mathfrak{S}({\mathcal{A}})$ be a Rauzy class. Since $\mathcal{R}$ commutes with dilations on $\R_{+}^{\mathcal{A}}$ it projectivizes to a map  $\mathcal{R}_R : \mathbb{P}_{+}^{\mathcal{A}} \times \mathfrak{R} \rightarrow \mathbb{P}_{+}^{\mathcal{A}} \times \mathfrak{R}$ called the \textit{Rauzy renormalization map}  which is defined in the complement of countably many hyperplanes.
Moreover we have that if $[\lambda] = [\lambda']$, then 
$B_R(\lambda',\pi)=B_R(\lambda,\pi)$ for any $\pi \in \mathfrak{R}$, hence the application $([\lambda],\pi) \rightarrow B_R([\lambda],\pi)$ is well defined. We refer to this cocycle as the Rauzy cocycle as well.

An induction scheme $\mathcal{S}: \R_{+}^{\mathcal{A}} \times \mathfrak{R} \rightarrow \R_{+}^{\mathcal{A}}  \times \mathfrak{R}$ is an \textit{acceleration of Rauzy induction} if there exists an integral application $m: \R_{+}^{\mathcal{A}} \times \mathfrak{R} \rightarrow \Z_+$, such that for every $(\lambda,\pi) \in \R_{+}^{\mathcal{A}} \times \mathfrak{R}$ we have $m(a\lambda,\pi)= m(\lambda,\pi)$ for all $a>0$ and
$$\mathcal{S}(\lambda,\pi) = \mathcal{R}^{m(\lambda,\pi)}(\lambda,\pi).$$

It is direct to see that $\mathcal{S}$ also commutes with dilations on $\R_{+}^{\mathcal{A}}$ and hence it projectivizes to a map  $\mathcal{S}_R : \mathbb{P}_{+}^{\mathcal{A}} \times \mathfrak{R} \rightarrow \mathbb{P}_{+}^{\mathcal{A}} \times \mathfrak{R}$ which we call an \textit{acceleration of Rauzy renormalization}.
Moreover we have that if $A:\mathbb{P}_{+}^{\mathcal{A}} \times \mathfrak{R} \rightarrow SL({\mathcal{A}},\Z)$ defines a cocycle over $\mathcal{S}$, then its projectivization $([\lambda],\pi) \rightarrow A([\lambda],\pi)$ is well defined.

A \textit{flag}, on  an $N$-dimensional vector space $F$, is a decreasing family of vector subspaces $\{F^j\}_{j=1,...,k+1}$, with $k\leq N$,
$$ F = F^1 \supsetneq F^2 \supsetneq ... \supsetneq F^k \supsetneq \{0\}= F^{k+1} .  $$
The flag is said to be \textit{complete} if $k=N$ and $\dim F^j = N+1-j$, for all $j=1,...,N$.

The following well known result follows from Oseledets Theorem \cite{O}.
\begin{theorem}\label{toseledets}
	Let $\mathfrak{R} \subseteq \mathfrak{S}({\mathcal{A}})$ be a Rauzy class, $\mathcal{S}_R: \mathbb{P}_{+}^{\mathcal{A}} \times \mathfrak{R} \rightarrow \mathbb{P}_{+}^{\mathcal{A}} \times \mathfrak{R}$ be an acceleration of Rauzy renormalization which is measurable with respect to an ergoric measure $m_{\mathfrak{R}}$ and let $A:\mathbb{P}_{+}^{\mathcal{A}} \times \mathfrak{R} \rightarrow SL({\mathcal{A}},\Z)$ be a $m_{\mathfrak{R}}$-measurable cocycle over $\mathcal{S}_R$.
	
	There exist $\kappa(\mathfrak{R}) \in \N$, real numbers $\nu_1(\mathfrak{R})>...>\nu_{\kappa(\mathfrak{R})}(\mathfrak{R})$ and for $m_{\mathfrak{R}}$-almost every $([\lambda],\pi) \in \mathbb{P}_{+}^{\mathcal{A}} \times \mathfrak{R}$ there exists a flag $ \R^{\mathcal{A}} = V_{[\lambda],\pi}^1  \supsetneq ... \supsetneq V_{[\lambda],\pi}^{\kappa(\mathfrak{R})} \supsetneq \{0\} = V_{[\lambda],\pi}^{\kappa(\mathfrak{R})+1}  $ such that $A([\lambda],\pi) \cdot V_{[\lambda],\pi}^{j}=V_{\mathcal{S}_R([\lambda],\pi)}^{j}$ and
	$$ \lim_{n\rightarrow +\infty}\frac{1}{n}\log \|  A^{(n)}([\lambda],\pi)\cdot v \| = \nu_j(\mathfrak{R}),  $$
	for all $v \in V_{[\lambda],\pi}^{j} \backslash V_{[\lambda],\pi}^{j+1}$, $j=1,...,\kappa(\mathfrak{R})$.
\end{theorem}

The spaces $V_{[\lambda],\pi}^j$ are called \textit{Oseledets subspaces} and the numbers $\nu_j(\mathfrak{R})$ are called the \textit{Lyapunov exponents} of the cocycle.
The integer $\dim V_{[\lambda],\pi}^j -\dim V_{[\lambda],\pi}^{j+1}$ is called the \textit{multiplicity} of the Lyapunov exponent $\nu_j(\mathfrak{R})$ and it is constant in a full measure set. The \textit{Lyapunov spectrum} of the cocycle is the set of its Lyapunov exponents counted with multiplicity.

In \cite{V1}, Veech proved that Rauzy renormalization admits an absolutely continuous ergodic measure. This measure, however is not finite and thus the Rauzy cocycle is not measurable with respect to it.

In \cite{Z1} Zorich defined an acceleration of Rauzy induction as follows. Given $(\lambda,\pi) \in \R_{+}^{\mathcal{A}} \times \mathfrak{S}({\mathcal{A}})$, let $n(\lambda,\pi)$ denote the smallest $n \in \N$ such that $\varepsilon(n) \neq \varepsilon(0)$ and set
$$  \mathcal{Z}(\lambda,\pi)= \mathcal{R}^{n(\lambda,\pi)}(\lambda,\pi).  $$
The map $\mathcal{Z}$ is called \textit{Zorich induction} and it projectivizes to a map $\mathcal{Z}_R : \mathbb{P}_{+}^{\mathcal{A}} \times \mathfrak{R} \rightarrow \mathbb{P}_{+}^{\mathcal{A}} \times \mathfrak{R}$ called \textit{Zorich renormalization}.

\begin{theorem}[\cite{Z1}]
	Let $\mathfrak{R}\subset \mathfrak{S}({\mathcal{A}})$ be a Rauzy class. Then $\mathcal{Z}_R :\mathbb{P}_{+}^{\mathcal{A}} \times \mathfrak{R}    \rightarrow \mathbb{P}_{+}^{\mathcal{A}} \times \mathfrak{R}$ admits a unique ergodic absolutely continuous probability measure $\mu_{\mathfrak{R}}$. Its density is positive and analytic.
\end{theorem}

Define the matrix function $B_{Z} : \R^{\mathcal{A}} \times \mathfrak{R} \rightarrow SL({\mathcal{A}},\Z)$ by
$$ B_{Z}(\lambda,\pi) = B(\lambda^{\left(  n(\lambda,\pi) -1 \right)},\pi^{\left(  n(\lambda,\pi) -1 \right)}) \cdot ... \cdot B(\lambda^{(1)},\pi^{(1)}) \cdot B(\lambda,\pi).$$

The \textit{Zorich cocycle} is the linear cocycle over the Zorich induction $(\mathcal{Z},B_{Z})$  on $\R_+^{\mathcal{A}} \times \mathfrak{R}\times \R^{\mathcal{A}}$. Its projectivization $(\mathcal{Z}_R,B_{Z})$ is well defined and also called Zorich cocycle.

Let $\| \cdot \|$ denote a matrix norm on $SL({\mathcal{A}},\Z)$ and let $\| A\|_{0}= \max\{ \| A \|,  \| A \|^{-1}   \}$ for any $A \in SL({\mathcal{A}},\Z)$. Recall we denote $\log^+y= \max\{\log(y) ,0  \}$ for any $y>0$.
\begin{theorem}[\cite{Z1}]\label{tzor}
	Let $\mathfrak{R}\subset \mathfrak{S}({\mathcal{A}})$ be a Rauzy class. Then
	$$  \int_{\mathbb{P}_{+}^{\mathcal{A}}\times \mathfrak{R} } \log^+ \| B_Z  \|_0 d\mu_{\mathfrak{R}} < + \infty. $$
	In particular $B_Z$ is a measurable cocycle with respect to $\mu_{\mathfrak{R}} $.
\end{theorem}

Recall the linear map $\Omega_{\pi}$ in \eqref{eqiet3}. Let $H_{\pi}$ be the image subspace of $\Omega_{\pi}$, that is, $H_{\pi}=\Omega_{\pi}(\R^{\mathcal{A}})$. From \cite{AF,V4} it follows that
\begin{equation}\label{BH}
 B(\lambda,\pi)\cdot H_{\pi} = H_{\pi^{(1)}},
\end{equation}
from which follows that $\dim H_{\pi}$ only depends on the Rauzy class $\mathfrak{R} \subset \mathfrak{S}({\mathcal{A}})$ of $\pi$. 

A \textit{translation surface} (as defined in \cite{AF}), is a surface with a finite number of conical singularities endowed with an atlas such that coordinate changes are given by translations in $\R^2$.
Given $(\lambda,\pi) \in \R_+^{\mathcal{A}} \times \mathfrak{R}$ it is possible (see for instance \cite{V1}) to associate, via a suspension construction, a translation surface, with genus $g(\mathfrak{R})\geq1$ and $\kappa$ singularities depending only on $\mathfrak{R}$. Moreover $\dim H_{\pi} = 2 g(\mathfrak{R})$.

By \eqref{BH}, it is direct to see that $H_{\pi}$ is an invariant subspace for both Rauzy and Zorich cocycles. Hence we can consider restrictions $B_R([\lambda],\pi)|_{H_\pi}$ and $B_Z([\lambda],\pi)|_{H_\pi}$ as integral cocycles over $\mathcal{R}_R$ and $\mathcal{Z}_R$ respectively, which we call \textit{restricted} Rauzy and Zorich cocycles. To simplify the notation we, at times, write $B_R([\lambda],\pi)$ and $B_Z([\lambda],\pi)$ instead of $B_R([\lambda],\pi)|_{H_\pi}$ and $B_Z([\lambda],\pi)|_{H_\pi}$.\\

As a consequence of theorems \ref{toseledets} and \ref{tzor}, for any Rauzy class $\mathfrak{R}\subset \mathfrak{S}({\mathcal{A}})$ there exist $k(\mathfrak{R}) \in \N$ such that for $\mu_{\mathfrak{R}}$-almost every $([\lambda],\pi) \in \mathbb{P}_{+}^{\mathcal{A}} \times \mathfrak{R}$ there exists a flag of Oseledets subspaces $ H_{\pi} = F_{[\lambda],\pi}^1  \supsetneq ... \supsetneq F_{[\lambda],\pi}^{k(\mathfrak{R})} \supsetneq \{0\} = F_{[\lambda],\pi}^{k(\mathfrak{R})+1}  $  with an associated Lyapunov spectrum $$\vartheta_{1}(\mathfrak{R})>...> \vartheta_{k(\mathfrak{R})}(\mathfrak{R}). $$ 


In \cite{Z1} it is shown that $k(\mathfrak{R}) \leq 2 g(\mathfrak{R})$ and that $\vartheta_j(\mathfrak{R})= - \vartheta_{k(\mathfrak{R})+1-j}(\mathfrak{R})$, for all $j =1,...,k(\mathfrak{R})$.
In \cite{AV} the authors proved that the Lyapunov spectrum of the restricted Zorich cocycle is \textit{simple} on every Rauzy class, that is, all Lyapunov exponents have multiplicity 1. Consequently, the spectral properties of the restricted Zorich cocycle can be summarized as follows.

\begin{theorem}\label{tspectralzor}
	Let $\mathfrak{R}\subset \mathfrak{S}({\mathcal{A}})$ be a Rauzy class. There exist Lyapunov exponents,
	\begin{equation*}\label{lspec}
	\vartheta_1(\mathfrak{R})>...>\vartheta_{g(\mathfrak{R})}(\mathfrak{R})>0>\vartheta_{g(\mathfrak{R})+1}(\mathfrak{R})=-\vartheta_{g(\mathfrak{R})}(\mathfrak{R})>...>\vartheta_{2g(\mathfrak{R})}(\mathfrak{R})=-\vartheta_{1}(\mathfrak{R}),
	\end{equation*} 
	 and, for $\mu_{\mathfrak{R}}$-almost every $([\lambda],\pi) \in \mathbb{P}_{+}^{\mathcal{A}}\times \mathfrak{R}$, there exists a complete flag $$ H_{\pi} = F_{[\lambda],\pi}^1  \supsetneq ... \supsetneq F_{[\lambda],\pi}^{2g(\mathfrak{R})} \supsetneq \{0\} = F_{[\lambda],\pi}^{2g(\mathfrak{R})+1},  $$ such that $B_{Z}([\lambda],\pi)|_{H_\pi} \cdot F_{[\lambda],\pi}^{j}=F_{\mathcal{Z}_R([\lambda],\pi)}^{j}$. For all $v \in F_{[\lambda],\pi}^{j} \backslash F_{[\lambda],\pi}^{j+1}$, $j=1,...,2g(\mathfrak{R})$,
	 $$ \lim_{n\rightarrow +\infty}\frac{1}{n}\log \|  B_{Z}([\lambda],\pi)|_{H_\pi}\cdot v \| = \vartheta_j(\mathfrak{R}).  $$
\end{theorem}

We say $([\lambda],\pi) \in \mathbb{P}_{+}^{\mathcal{A}}\times \mathfrak{R}$ is \textit{generic} if $([\lambda],\pi)$ is in the full measure set of $\mathbb{P}_{+}^{\mathcal{A}}\times \mathfrak{R}$ from Theorem \ref{tspectralzor}.

Let $\|\cdot\|_1:SL({\mathcal{A}},\Z)\rightarrow \R_+$ be the norm,
$$\|A\|_1= \sum_{\alpha \in {\mathcal{A}}}\sum_{\beta \in {\mathcal{A}}} |A_{\alpha \beta}|.$$

Denote by $\textrm{Leb}$ the Lebesgue measure in $\mathbb{P}_{+}^{\mathcal{A}}$ and by $c_{\mathfrak{R}}$ the counting measure in a Rauzy class $\mathfrak{R}$. The following theorem is a restatement of a result by Marmi, Moussa and Yoccoz \cite{MMY} and gives a bound for the growth of the Zorich cocycle for a full measure set of $([\lambda],\pi)$. The proof can be found in Section 4.7 in \cite{MMY}.
\begin{theorem}[\cite{MMY}]\label{tMMY}
	For $\textrm{Leb} \times c_{\mathfrak{R}}$-almost every $([\lambda],\pi) \in \mathbb{P}_+^{\mathcal{A}} \times \mathfrak{R}$ and $\varepsilon'>0$, there exists $C_{\varepsilon'}>0$ such that for any $m\geq0$,
	$$ \|  B_{Z}\left( \mathcal{Z}_R^{m}([\lambda],\pi)    \right)  \|_1 < C_{\varepsilon'} \|  B_{Z}^{(m)}\left(  [\lambda],\pi    \right)  \|_1^{\varepsilon'} $$
\end{theorem}

Given $([\lambda],\pi) \in \mathbb{P}_{+}^{\mathcal{A}}\times \mathfrak{R}$ and $m \geq 0$, denote the sum of the $m$ first Zorich acceleration times by
$$ s^m([\lambda],\pi)  = \sum_{k<m} n (\mathcal{Z}_R^{k}([\lambda],\pi)).$$

So far the choice of vector norm $\| \cdot \|$ has not been relevant as Theorem \ref{toseledets} does not depend on any particular choice. However in what follows we consider $\| \cdot \|$ to be the euclidean norm.

In the following lemma we combine estimates from theorems \ref{tspectralzor} and \ref{tMMY} to obtain an important bound for the growth of the Rauzy cocycle, restricted to $F_{[\lambda],\pi}^{g(\mathfrak{R})+1} \backslash \{0\}$, for a full measure set of parameters.

\begin{lemma}\label{lsumrauzy}
	 For $\textrm{Leb} \times c_{\mathfrak{R}}$-almost every $([\lambda],\pi) \in \mathbb{P}_+^{\mathcal{A}} \times \mathfrak{R}$, there exists $K\geq 1$ such that for all $v \in F_{[\lambda],\pi}^{g(\mathfrak{R})+1} \backslash \{0\}$ we have
	\begin{equation}\label{SR'}
	\sum_{n=0}^{+\infty} \|B_{R}^{(n)}([\lambda],\pi)\cdot v \| < K \| v  \|.
	\end{equation}
\end{lemma}

\begin{proof}
	
	By Theorem \ref{tspectralzor}, for $\mu_{\mathfrak{R}}$-almost every $([\lambda],\pi) \in \mathbb{P}_+^{\mathcal{A}} \times \mathfrak{R}$ and any $0<\eta<1$ there exists $K_{\eta}>0$ such that for every $m \geq 0$,
	\begin{equation*}\label{SR1}
	\|  B_{Z}^{(m)}\left(  [\lambda],\pi    \right)  \|_1< K_{\eta} e^{\eta^{-1}\vartheta_1(\mathfrak{R}) m}
	\end{equation*}
	As, by Theorem \ref{tzor}, $\mu_{\mathfrak{R}}$ has positive density, this also holds for $\textrm{Leb} \times c_{\mathfrak{R}}$-a.e. $([\lambda],\pi) \in \mathbb{P}_+^{\mathcal{A}} \times \mathfrak{R}$. Combined with Theorem \ref{tMMY}, for $\varepsilon'=\frac{1}{4}\eta^2 \vartheta_{g(\mathfrak{R})}(\mathfrak{R}) / \vartheta_1(\mathfrak{R}) $, this gives
	\begin{equation}\label{SR2}
	\|  B_{Z}\left(  \mathcal{Z}_R^{m}([\lambda],\pi)    \right)  \|_1< K_{\eta} C_{\varepsilon'} e^{\frac{1}{4}\eta\vartheta_{g(\mathfrak{R})}(\mathfrak{R}) m},
	\end{equation}
	for $\textrm{Leb} \times c_{\mathfrak{R}}$-a.e. $([\lambda],\pi) \in \mathbb{P}_+^{\mathcal{A}} \times \mathfrak{R}$.
	
	By Theorem \ref{tspectralzor} we also get that for $\textrm{Leb} \times c_{\mathfrak{R}}$-a.e. $([\lambda],\pi) \in \mathbb{P}_+^{\mathcal{A}} \times \mathfrak{R}$ there exists $K_{\eta}'>0$, such that, for any $v \in F_{[\lambda],\pi}^{g(\mathfrak{R})+1} \backslash \{0\}$ we have
	\begin{equation}\label{SR3}
	\|  B_{Z}^{(m)}\left(  [\lambda],\pi    \right) \cdot v  \|< K_{\eta}' e^{-\eta\vartheta_{g(\mathfrak{R})}(\mathfrak{R}) m} \| v \|.
	\end{equation}
	Let $\mathcal{E}_{\eta}$ denote the set of $([\lambda],\pi) \in \mathbb{P}_+^{\mathcal{A}} \times \mathfrak{R}$ for which there exists $K_{\eta}''>0$ such that
	\begin{equation}\label{SR4}
	\|  B_{Z}\left(  \mathcal{Z}_R^{m}([\lambda],\pi)    \right)  \|_1^{2} \cdot \|  B_{Z}^{(m)}\left(  [\lambda],\pi    \right) \cdot v  \| < K_{\eta}'' e^{-\frac{1}{2} \eta\vartheta_{g(\mathfrak{R})}(\mathfrak{R}) m} \| v \|,
	\end{equation}
	for all $v \in F_{[\lambda],\pi}^{g(\mathfrak{R})+1} \backslash \{0\}$ and $m\geq 0$. By combining \eqref{SR2} and \eqref{SR3} we get that $\mathcal{E}_{\eta} $ is a set of full $\textrm{Leb} \times c_{\mathfrak{R}}$ measure.\\


	Now, fix $0<\eta<1$ and $ ([\lambda],\pi) \in \mathcal{E}_{\eta} $. For $n \geq 0$, let $ M(n) =  \max\left\{  m \geq 0  : s^m([\lambda],\pi) \leq n      \right\}$. Also, given positive integers $k_1<k_2$ we denote
	$$B^{(k_1,k_2)}([\lambda],\pi) =  B([\lambda^{(k_2)}],\pi^{(k_2)}) \cdot B([\lambda^{(k_2-1)}],\pi^{(k_2-1)}) \cdot ... \cdot B([\lambda^{(k_1)}],\pi^{(k_1)}).  $$
	We have
	\begin{equation}\label{SR5}
	\| B_{R}^{(n)}([\lambda],\pi) \cdot v \|  \leq \max_{s^{M(n)}([\lambda],\pi) \leq k < n} \left \|  B^{\left(s^{M(n)}([\lambda],\pi),k\right)}([\lambda],\pi)    \cdot B_Z^{(  M(n) )}  ([\lambda],\pi) \cdot v   \right \| .
	\end{equation}
	It is clear that we have
	$$   \max_{s^{M(n)}([\lambda],\pi) \leq k < n}    \left\|  B^{\left(s^{M(n)}([\lambda],\pi),k\right)}([\lambda],\pi) \right\|_1      \leq \left \|   B_Z \left( \mathcal{Z}_R^{ M(n) }([\lambda],\pi)     \right)     \right \|_1   , $$
	hence, from \eqref{SR5}, for all $n \geq 0$ we get
	$$\| B_{R}^{(n)}([\lambda],\pi) \cdot v \| \leq  \|  B_{Z}(  \mathcal{Z}_R^{M(n)}([\lambda],\pi)    )  \|_1 \cdot \|  B_{Z}^{(M(n))}\left(  [\lambda],\pi    \right) \cdot v  \|,$$
	which combined with the fact that for all $m \geq 0$ we have $n(\mathcal{Z}_R^m([\lambda],\pi)) \leq \|  B_{Z}\left(  \mathcal{Z}_R^{m}([\lambda],\pi)    \right)  \|_1$,  gives
	$$  \sum_{n=0}^{+\infty} \|B_{R}^{(n)}([\lambda],\pi)\cdot v \|   \leq \sum_{m=0}^{+\infty} \|  B_{Z}\left(  \mathcal{Z}_R^{m}([\lambda],\pi)    \right)  \|_1^{2} \cdot \|  B_{Z}^{(m)}\left(  [\lambda],\pi    \right) \cdot v  \| . $$
	This, combined with \eqref{SR4}, which holds since $([\lambda],\pi) \in \mathcal{E}_{\eta}$, shows that by taking $K=\max\{K_{\eta}'' (1-e^{-1/2\eta \vartheta_{g(\mathfrak{R})}(\mathfrak{R})})^{-1},1\}$ we get \eqref{SR'} as intended.
\end{proof}


Recalling \eqref{crauzytor} note that for any $\lambda, \lambda' \in \R_{+}^{\mathcal{A}}$ such that $[\lambda] = [\lambda']$ we have $B_{\T^{\mathcal{A}}}(\lambda,\pi)= B_{\T^{\mathcal{A}}}(\lambda',\pi)$ and thus $B_{\T^{\mathcal{A}}}(\lambda,\pi)$ admits a projectivization which we denote $B_{\T^{\mathcal{A}}}([\lambda],\pi)$ and also call projection of the Rauzy cocycle on $\T^{\mathcal{A}}$.

Recall that $p:\R^{\mathcal{A}} \rightarrow  \mathbb{T}^{\mathcal{A}}$ is the natural projection,
$$ p(v) = \left(  \left( v    \right)_{\alpha}\mod 2 \pi    \right)_{\alpha \in \mathcal{A}}, \quad \textrm{ for all} \ v \in \R^{\mathcal{A}}.$$

The \textit{flat torus} is the torus $\T^{\mathcal{A}}$ viewed as a Riemannian manifold equipped with the \textit{flat Riemannian metric}, this is, the pushforward under $p$ of the euclidean metric in $\R^{\mathcal{A}}$.
The flat Riemannian metric induces a distance on the torus $d_{\T^{\mathcal{A}}}: \T^{\mathcal{A}} \times \T^{\mathcal{A}} \rightarrow \R_+$ such that
$$d_{\T^{\mathcal{A}}}(\theta,\theta') = \inf \left \{     \| v-v' \| : v \in p^{-1}(\theta), \ v' \in p^{-1}(\theta')  \right \} , $$
for any $\theta, \theta' \in \T^{\mathcal{A}}$.

Given $\delta>0$ and  a generic $([\lambda],\pi) \in \mathbb{P}_+^{\mathcal{A}} \times \mathfrak{R}$, let
$$  E_{[\lambda],\pi}^{\delta} = \left\{  v \in  F_{[\lambda],\pi}^{g(\mathfrak{R})+1} \backslash \{0\} : \| v \| <\delta    \right\} ,$$
and let $ W_{[\lambda],\pi}^{\delta} = p \left(  E_{[\lambda],\pi}^{\delta}  \right) $.

 Recall \eqref{thetan}, which given $\theta \in \T^{\mathcal{A}}$ defines a sequence $\{\theta^{(n)}\}_{n\geq 0}$ on $\mathbb{T}^{\mathcal{A}}$ which is used to construct the breaking sequence $\{\gamma_{\theta}^{(n)}\}_{n\geq 0}$.
The following lemma states that for a full measure set of $([\lambda],\pi)$, and for sufficiently small $\delta>0$, and all $\theta \in W_{[\lambda],\pi}^{\delta}$ the sum of all $d_{\T^{\mathcal{A}}}(\theta^{(n)},0)$ is bounded.

\begin{lemma}\label{lsumtor}
	For $\textrm{Leb} \times c_{\mathfrak{R}}$-almost every $([\lambda],\pi) \in \mathbb{P}_+^{\mathcal{A}} \times \mathfrak{R}$, there exists $K\geq1$ and $\delta>0$ such that for all $\theta \in W_{[\lambda],\pi}^{\delta}$ we have
	\begin{equation}\label{ST'}
	\sum_{n=0}^{+\infty}  d_{\T^{\mathcal{A}}}(\theta^{(n)},0) < K d_{\T^{\mathcal{A}}}(\theta,0).
	\end{equation}
\end{lemma}

\begin{proof}
By Lemma \ref{lsumrauzy}, for $\textrm{Leb} \times c_{\mathfrak{R}}$-a.e. $([\lambda],\pi) \in \mathbb{P}_+^{\mathcal{A}} \times \mathfrak{R}$, there exists $K>1$ such that for all $v \in E_{[\lambda],\pi}^{\delta}$, with $\delta = \pi \cdot K^{-1}$, for all $n\geq 0$ we have
\begin{equation*}\label{}
\| B_R^{(n)}([\lambda],\pi)\cdot v \| <\pi,
\end{equation*}

Moreover, is is clear that if $\| v\| < \pi$ we have $d_{\T^{\mathcal{A}}}(p(v),0) = \| v  \|$, thus, for all $n\geq 0$ we have
\begin{equation}\label{ST1}
d_{\T^{\mathcal{A}}}\left(p  \left( B_R^{(n)}([\lambda],\pi)\cdot v   \right),0\right)  =  \left \| B_R^{(n)}([\lambda],\pi)\cdot v   \right \|,
\end{equation}
Also note that as $\delta \leq \pi$,  the restriction $p|_{E_{[\lambda],\pi}^{\delta}} : E_{[\lambda],\pi}^{\delta} \rightarrow W_{[\lambda],\pi}^{\delta}$ is a bijection and thus $p^{-1}(\theta) \cap E_{[\lambda],\pi}^{\delta}$ contains a single point which we denote by $p_{\delta}^{-1}(\theta)$. Take $\theta \in W_{[\lambda],\pi}^{\delta}$.
It is clear by \eqref{crauzytor} that we have
$$  B_{\T^{\mathcal{A}}}^{(n)}([\lambda],\pi)  \cdot \theta   =  p \left(   B_{R}^{(n)}([\lambda],\pi) \cdot  p_{\delta}^{-1}(\theta)   \right) ,       $$
which combined with \eqref{ST1} yields $   d_{\T^{\mathcal{A}}}(B_{\T^{\mathcal{A}}}^{(n)}([\lambda],\pi)  \cdot \theta ,0)   =  \|   B_{R}^{(n)}([\lambda],\pi) \cdot  p_{\delta}^{-1}(\theta) \|$, for all $n \geq 0$. By \eqref{thetan} and Lemma \ref{lsumrauzy} this gives \eqref{ST'} finishing our proof.
\end{proof}

We say a map $\gamma : I \rightarrow \C$ is Lipschitz if $\{(\Re(\gamma(x)), \Im(\gamma(x))) : x \in I   \}$ is the graph of a Lipschitz map. The following theorem shows that for a generic $([\lambda],\pi)$ and sufficiently small $\delta>0$, when $\theta \in W_{[\lambda],\pi}^{\delta}$  the sequence $\gamma_{\theta}^{(n)}$ converges to a a Lipschitz map $\gamma_{\theta}$ which is an isometric embedding of $(I,f_{\lambda,\pi})$ into any PWI that is $\theta$-adapted to $(\lambda,\pi)$.
\begin{theorem}\label{taexistembeds}
	For $\textrm{Leb} \times c_{\mathfrak{R}}$-almost every $([\lambda],\pi) \in \mathbb{P}_+^{\mathcal{A}} \times \mathfrak{R}$, there exists a $\delta>0$ such that for all $\theta \in W_{[\lambda],\pi}^{\delta}$ there exists a Lipschitz map $\gamma_{\theta} : I \rightarrow \C$, which is an isometric embedding of $(I,f_{\lambda,\pi})$ into any PWI that is $\theta$-adapted to $(\lambda,\pi)$.
\end{theorem}

\begin{proof}
	Consider the space $\mathcal{C}(I,\C)$ of continuous maps from the interval $I$, to $\C$. Note that this is a Banach space for the supremum norm $\|.\|_{\infty}$. We also have that $\gamma_{\theta}^{(n)} \in \mathcal{C}(I,\C)$ for all $n \geq 0$, since $\gamma_{\theta}^{(0)}$ is continuous and, by Lemma \ref{lPL}, $\mathfrak{Br}\left(  \theta_{\beta_{1,n-1}}^{(n-1)} ,  J^{(n)}\right) \cdot \mathcal{C}(I,\C)\subseteq \mathcal{C}(I,\C)$.
	
	Take any $\varphi \in (0, \pi/2 )$. By Lemma \ref{lsumtor}, there exists a set $\mathcal{E}\subseteq \mathbb{P}_+^{\mathcal{A}} \times \mathfrak{R}$ of full $\textrm{Leb} \times c_{\mathfrak{R}}$ measure such that for every $([\lambda],\pi) \in \mathcal{E}$, there exists $K\geq1$ and $0<\delta< \varphi K^{-1}$ such that for all $\theta \in W_{[\lambda],\pi}^{\delta}$ we have \eqref{ST'}.
	
	Take $([\lambda],\pi) \in \mathcal{E}$ and $\theta \in W_{[\lambda],\pi}^{\delta}$. For all  $x \in I$ we have
	\begin{equation*}\label{aexist5}
	\left| \gamma_{\theta}^{(n+1)}(x)-\gamma_{\theta}^{(n)}(x)\right|= \left| \mathfrak{Br}\left( \theta_{\beta_{1,n}}^{(n)},  J^{(n+1)}\right) \cdot \gamma_{\theta}^{(n)}(x)-\gamma_{\theta}^{(n)}(x)       \right|.
	\end{equation*}
	Denoting, as in \eqref{Jn}, by $r(n)$ the number of intervals of $J^{(n+1)}$, by \eqref{breq1} this gives
	\begin{equation*}\label{aexist6}
	\left| \gamma_{\theta}^{(n+1)}(x)-\gamma_{\theta}^{(n)}(x)\right|\leq \max_{k < r(n)}\left\{ |\underline{\epsilon}_k| , |\overline{\epsilon}_k| \right\} +\sup_{x \in I}|\gamma_{\theta}^{(n)}(x)(1-e^{i \theta_{\beta_{1,n}}^{(n)} })|.
	\end{equation*}
	Since
	\begin{equation*}
	\sup_{x \in I}\left|  \gamma_{\theta}^{(n)}(x)(1-e^{i \theta_{\beta_{1,n}}^{(n)} })     \right| \leq 2 |\lambda|\sin\left( \theta_{\beta_{1,n}}^{(n)}  /2\right),
	\end{equation*}
	by Lemma \ref{brth1} we get
	\begin{equation*}\label{aexist7}
	\left| \gamma_{\theta}^{(n+1)}(x)-\gamma_{\theta}^{(n)}(x)\right|\leq 4|\lambda|\sin\left( \theta_{\beta_{1,n}}^{(n)}/2  \right).
	\end{equation*}
	Therefore, as $\theta_{\beta_{1,n}}^{(n)} \leq d_{\T^{\mathcal{A}}}(\theta^{(n)},0)$  there exists $C>0$ such that for all $n\geq 0$,
	\begin{equation*}
	\left| \gamma_{\theta}^{(n+1)}(x)-\gamma_{\theta}^{(n)}(x)\right|\leq C |\lambda|  d_{\T^{\mathcal{A}}}(\theta^{(n)},0).
	\end{equation*}
	Now take $m$, $n \in \N$ such that $m>n$. Note that we have
	\begin{equation*}
	\| \gamma_{\theta}^{(m)}- \gamma_{\theta}^{(n)} \|_{\infty}\leq \sum_{k=0}^{m-n-1}\| \gamma_{\theta}^{(m-k)}- \gamma_{\theta}^{(m-k-1)} \|_{\infty},
	\end{equation*} 
	and therefore
	\begin{equation}\label{aexist8}
	\| \gamma_{\theta}^{(m)}- \gamma_{\theta}^{(n)} \|_{\infty}\leq C |\lambda| \sum_{k=n}^{m-1}  d_{\T^{\mathcal{A}}}(\theta^{(k)},0),
	\end{equation} 
	From \eqref{ST'} by taking a sufficiently large $N>0$ and considering $N<n<m$ the righthand side  of \eqref{aexist8} can be made arbitrarily small. Thus $\{ \gamma_{\theta}^{(n)} \}_{n \geq 0}$ is a Cauchy sequence in  $\mathcal{C}(I,\C)$ and therefore it must converge to a unique limit $\gamma_{\theta} \in \mathcal{C}(I,\C)$.\\

	As for all $n \geq 0$, $\gamma_{\theta}^{(n)} \in \mathcal{C}(I,\C)$, by \eqref{gamman} it is simple to see that for any $x, y \in I$, $x \neq y$, we have
	$$\frac{  \left|  \Im(\gamma_{\theta}^{(n)}(x)) -  \Im(\gamma_{\theta}^{(n)}(y))  \right| }{  \left|  \Re(\gamma_{\theta}^{(n)}(x)) -  \Re(\gamma_{\theta}^{(n)}(y)) \right|} \leq \tan \left(     \sum_{n=1}^{+ \infty} \theta_{\beta_{1,n-1}}^{(n-1)}     \right) . $$
	
	For any map $\gamma:I\rightarrow \C$, its Lipschitz constant $L(\gamma)$ is given by
	$$ L(\gamma) = \sup_{x, y \in I, \ x\neq y}  \frac{  \left|  \Im(\gamma(x)) -  \Im(\gamma(y))  \right| }{  \left|  \Re(\gamma(x)) -  \Re(\gamma(y)) \right|} . $$
	Hence, in particular we get,	
	$$  \arctan(L(\gamma_{\theta}^{(n)}))\leq  \sum_{n=0}^{+ \infty} d_{\T^{\mathcal{A}}}(\theta^{(n)},0), $$
	which, as $\delta < \varphi K^{-1}$,  by \eqref{ST'} gives $ \arctan(L(\gamma_{\theta}^{(n)}))\leq  \varphi $.
	Clearly $L(\gamma_{\theta} )\leq  \sup_{n\geq 0}L(\gamma_{\theta}^{(n)})$, and as $\varphi<\pi/2$ this shows that $L(\gamma_{\theta} )< + \infty$ and thus $\gamma_{\theta}$ is a Lipschitz map. In particular it is continuous and injective and thus a topological embedding.
	
	This proves that $W_{[\lambda],\pi}^{\delta}\subseteq  \Theta_{\lambda,\pi}' $ and therefore by Theorem \ref{tOmembed}, for any $\theta \in W_{[\lambda],\pi}^{\delta}$, $\gamma_{\theta}$ is an isometric embedding of $(I,f_{\lambda,\pi})$ into any PWI that is $\theta$-adapted to $(\lambda,\pi)$.
\end{proof}


As described in \cite{AGPR}, we can extend Rauzy-Veech induction to PWIs which admit embeddings of IETs as follows. Assume $(I,f_{\lambda,\pi})$ has an embedding by $\gamma_{\theta}$ into $(X,T)$. Define the map $\mathcal{S}(T)$ as the first return map under $T$ to $X^*$, where
\begin{equation*}
X^*=\left\{\begin{array}{ll}
\vspace{0.2cm}
\bigcup_{\alpha \neq \beta_0} X_{\alpha} \cup (X_{\beta_{0}} \cap T(X_{\beta_{1}})) , & \textrm{if} \ (\lambda,\pi) \ \textrm{has type} \ 0,\\
\bigcup_{\alpha \neq \beta_0} X_{\alpha} , & \textrm{if} \ (\lambda,\pi) \ \textrm{has type} \ 1.
\end{array}\right.
\end{equation*}
It is clear that $(X^*,\mathcal{S}(T))$ is again a $d'$-PWI, with possibly $d' \neq d$. Denote by $\mathcal{A}'$ an alphabet with $d'$ symbols and denote by $\{X_{\alpha'}^*\}_{\alpha' \in \mathcal{A}'}$ the partition of $X^*$. It is simple to see that there is a collection of $d$ symbols $\mathcal{A} \subseteq \mathcal{A}'$, possibly after relabeling, such that $X_{\alpha'}^* \cap \gamma_{\theta}(I^{(1)}) \neq \emptyset$ if and only if $\alpha' \in \mathcal{A}$. Define $X'= \bigcup_{\alpha \in \mathcal{A}}X_{\alpha}^*$.

Now, $(X',\mathcal{S}(T))$ is $\theta^{(1)}$-adapted to $(\lambda^{(1)},\pi^{(1)})$ and, by Theorem 2.3 in \cite{AGPR}, the restriction of $\gamma_{\theta}$ to $I^{(1)}$ is an embedding of $(I^{(1)},f_{\lambda^{(1)},\pi^{(1)}})$ into $(X',\mathcal{S}(T))$. 

It is thus possible to iterate this procedure by setting $(X^{(0)}, \mathcal{S}^{(0)}(T)) = (X,T)$, and  $\left( X^{(n)}, \mathcal{S}^{(n)}(T) \right) = \left( (X^{(n-1)})', \mathcal{S}(\mathcal{S}^{(n-1)}(T)) \right)$ for $n \geq 1$. The following lemma easily follows from Theorem \ref{tOmembed}.
\begin{lemma}\label{LS(T)}
	Let $(\lambda,\pi)\in \R_+^{\mathcal{A}}\times \mathfrak{R}$, $\theta \in \Theta_{\lambda,\pi}'$ and $(X,T)$ be a PWI $\theta$-adapted to $(\lambda,\pi)$.
	Then for all $n\geq 0$, $(X^{(n)},\mathcal{S}^{(n)}(T))$ is $\theta^{(n)}$-adapted to $(\lambda^{(n)},\pi^{(n)})$ and the restriction of $\gamma_{\theta}$ to $I^{(n)}$ is an embedding of $(I^{(n)},f_{\lambda^{(n)},\pi^{(n)}})$ into $(X^{(n)},\mathcal{S}^{(n)}(T))$. 
\end{lemma}

%

Given a generic $([\lambda],\pi) \in \mathbb{P}_+^{\mathcal{A}} \times \mathfrak{R}$ and $\delta>0$ note that $W_{[\lambda],\pi}^{\delta}$ defines a $g(\mathfrak{R})$-dimensional submanifold embedded in the torus $\T^{\mathcal{A}}$. Pulling back the flat metric by the embedding map it is possible to construct a $g(\mathfrak{R})$-volume form and thus define a positive measure $m_{g(\mathfrak{R})}$ on $W_{[\lambda],\pi}^{\delta}$.

Denote the projection on $\T^{\mathcal{A}}$ of the Oseledets subspace $F_{[\lambda],\pi}^{2g(\mathfrak{R})}$ by $W_{[\lambda],\pi}^{SS}   =  p\left( F_{[\lambda],\pi}^{2g(\mathfrak{R})}  \right) $. Note that $W_{[\lambda],\pi}^{SS}$ is a $1$-dimensional submanifold embedded in $\T^{\mathcal{A}}$.

For any $n\geq 0$ and $\theta \in \T^{\mathcal{A}}$ let $B_{\T^{\mathcal{A}}}^{(-n)}([\lambda],\pi)\cdot \theta = \left\{  \theta' \in  \T^{\mathcal{A}} : B_{\T^{\mathcal{A}}}^{(n)}([\lambda],\pi)\cdot \theta' =  \theta   \right\}.$
Consider
$$  \mathcal{W}_{[\lambda],\pi}^{\delta} = W_{[\lambda],\pi}^{\delta} \backslash \left(  W_{[\lambda],\pi}^{SS} \cup \bigcup_{n=0}^{+\infty} B_{\T^{\mathcal{A}}}^{(-n)}([\lambda],\pi)\cdot 0    \right).  $$

Recall the definitions of arc, linear and non-trivial embeddings in the Introduction.
The following theorem establishes that for any Rauzy class $\mathfrak{R}$ such that $g(\mathfrak{R})\geq 2$ and for a full measure set of $([\lambda],\pi) \in \mathbb{P}_+^{\mathcal{A}} \times \mathfrak{R}$, when $\theta \in \mathcal{W}_{[\lambda],\pi}^{\delta} $ for sufficiently small $\delta >0$, $\gamma_{\theta}$ is a non-trivial isometric embedding of $(I,f_{\lambda,\pi})$ into any PWI $(X,T)$ that is $\theta$-adapted to $(\lambda,\pi)$. 
Since $(I,f_{\lambda,\pi})$ is topologically conjugated to the restriction of $(X,T)$ to the image of the embedding $\gamma_{\theta}(I)$ we have that the latter map is one-to-one and therefore $\gamma_{\theta}(I)$ is an invariant set for $(X,T)$. Moreover $\gamma_{\theta}(I)$ is a curve which is not a union of line segments or circle arcs. Thus, Theorem \ref{TA} follows directly from Theorem \ref{taeexistnontrivembeds} which we now state and prove.

\begin{theorem}\label{taeexistnontrivembeds}
	For any Rauzy class $\mathfrak{R}$ satisfying $g(\mathfrak{R})\geq 2$ and $\textrm{Leb} \times c_{\mathfrak{R}}$-almost every $([\lambda],\pi) \in \mathbb{P}_+^{\mathcal{A}} \times \mathfrak{R}$, there exists  $\delta>0$ such that  $\mathcal{W}_{[\lambda],\pi}^{\delta}$ is a set of full $m_{g(\mathfrak{R})}$-measure in $W_{[\lambda],\pi}^{\delta}$ and for all $\theta \in  \mathcal{W}_{[\lambda],\pi}^{\delta}$ there exists a Lipschitz map $\gamma_{\theta} : I \rightarrow \C$, which is a non-trivial isometric embedding of $(I,f_{\lambda,\pi})$ into any PWI that is $\theta$-adapted to $(\lambda,\pi)$.
\end{theorem}

\begin{proof}

	As for any $\delta>0$ we have $ \mathcal{W}_{[\lambda],\pi}^{\delta} \subseteq W_{[\lambda],\pi}^{\delta}$, by Theorem \ref{taexistembeds} for $\textrm{Leb} \times c_{\mathfrak{R}}$-almost every $([\lambda],\pi) \in \mathbb{P}_+^{\mathcal{A}} \times \mathfrak{R}$, there exists  $\delta>0$ such that for all $\theta \in \mathcal{W}_{[\lambda],\pi}^{\delta}$ there exists a Lipschitz map $\gamma_{\theta} : I \rightarrow \C$, which is an isometric embedding of $(I,f_{\lambda,\pi})$ into any PWI that is $\theta$-adapted to $(\lambda,\pi)$.

	Note that $\bigcup_{n=0}^{+\infty} B_{\T^{\mathcal{A}}}^{(-n)}([\lambda],\pi)\cdot 0$ is a countable set, $\dim(W_{[\lambda],\pi}^{SS})  = 1$ and $\dim(W_{[\lambda],\pi}^{\delta})= g(\pi)$. Thus, when $g(\mathfrak{R}) \geq 2$ we have that $\mathcal{W}_{[\lambda],\pi}^{\delta}$ is a set of full $m_{g(\mathfrak{R})}$-measure in $W_{[\lambda],\pi}^{\delta}$.

	For $\theta \in \mathcal{W}_{[\lambda],\pi}^{\delta}$, assume by contradiction that $\gamma_{\theta}$ is an arc embedding of $(I,f_{\lambda,\pi})$ into a PWI $(X,T)$ that is $\theta$-adapted to $(\lambda,\pi)$.   
	There exists $x'>0$ such that the restriction of $\gamma_{\theta}$ to $[0,x')$ is an arc map. Moreover, there exists an $N \in \N$ such that for all $n\geq N$ we have $I^{(n)} \subseteq [0,x')$. As $\gamma_{\theta}$ is an isometric embedding and $\gamma_{\theta}(0) = 0$, there is an $r>0$ and a $\varphi \in [0,2\pi)$ such that for all $x \in I^{(n)}$ we have
	\begin{equation}\label{ent3}
	\gamma_{\theta}(x) = r ( e^{i (r^{-1} x + \varphi)} - e^{i  \varphi}).
	\end{equation}
	
	By Lemma \ref{LS(T)}, for any $n \geq N$, $(X^{(n)},\mathcal{S}^{(n)}(T))$ is a PWI $\theta^{(n)}$-adapted to $(\lambda^{(n)},\pi^{(n)})$ and the restriction of $\gamma_{\theta}$ to $I^{(n)}$ is an isometric embedding of $(I^{(n)},f_{\lambda^{(n)},\pi^{(n)}})$ into $(X^{(n)},\mathcal{S}^{(n)}(T))$. Hence we have
	\begin{equation}\label{ent4}
	\gamma_{\theta}(f_{\lambda^{(n)},\pi^{(n)}}(x)) = e^{i \theta_{\alpha}^{(n)}} \left( \gamma_{\theta}(x) -\gamma_{\theta}\left(x_{\pi_{0}^{(n)}(\alpha)-1}^{(n)} \right)   \right) + \gamma_{\theta} \left(f_{\lambda^{(n)},\pi^{(n)}}\left(x_{\pi_{0}^{(n)}(\alpha)-1}^{(n)} \right) \right), 
	\end{equation}
	for all $\alpha \in \mathcal{A}$, any $x \in I_{\alpha}^{(n)}$ and any $n \geq N$.
		
	Recall that we denote $\upsilon^{( n  )} = \Omega_{\pi^{( n  )}}(\lambda^{(n )})$. Let $M>0$ be such that for all $m\geq M$ we have $s^m([\lambda],\pi)>N$.	From \eqref{ent3}, \eqref{ent4} and \eqref{eqiet2} we have
	\begin{equation}\label{ent5}
	\theta^{( s^m([\lambda]\pi)  )}= p \left ( r^{-1} \upsilon^{(  s^m([\lambda],\pi)   )} \right).
	\end{equation}
	By the proof of Theorem \ref{taexistembeds} we have $\delta < \pi$ and thus, the restriction $p|_{E_{[\lambda],\pi}^{\delta}} : E_{[\lambda],\pi}^{\delta} \rightarrow W_{[\lambda],\pi}^{\delta}$ is a bijection and thus $p^{-1}(\theta) \cap E_{[\lambda],\pi}^{\delta}$ contains a single point which we denote by $p_{\delta}^{-1}(\theta)$. As $\theta \in \mathcal{W}_{[\lambda],\pi}^{\delta}$, by \eqref{ent5} we get
	\begin{equation}\label{ent5a}
	\upsilon^{(  s^m([\lambda],\pi)  )} = B_Z^{(m)}([\lambda],\pi)\cdot p_{\delta}^{-1}(\theta). 
	\end{equation}
	By the results in \cite{Vi07} Section 5.3, it is known that $F_{[\lambda],\pi}^{2g(\mathfrak{R})}$ is equal to the linear span of $\{\upsilon^{(0)}\}$ in $\R^{\mathcal{A}}$ and thus by \eqref{ent5a} and Theorem \ref{tspectralzor} we get that $p_{\delta}^{-1}(\theta) \in F_{[\lambda],\pi}^{2g(\mathfrak{R})}$ and consequently $\theta \in W_{[\lambda],\pi}^{SS}$ which contradicts our assumption $\theta \in \mathcal{W}_{[\lambda],\pi}^{\delta}$. Therefore $\gamma_{\theta}$ is not an arc embedding.
	
	
	Now, for $\theta \in \mathcal{W}_{[\lambda],\pi}^{\delta} $, assume by contradiction that $\gamma_{\theta}$ is a linear embedding of $(I,f_{\lambda,\pi})$ into a PWI $(X,T)$ that is $\theta$-adapted to $(\lambda,\pi)$.  
	As $\gamma_{\theta}$ is an isometric embedding and $\gamma_{\theta}(0)=0$ for a sufficiently large $N \in \N$ there is  $\varphi \in [0,2\pi)$ such that
	\begin{equation}\label{ent6}
	\gamma_{\theta}(x) = e^{i\varphi}x,
	\end{equation}
	for all $x \in I^{(N)}$.
	
	By Lemma \ref{LS(T)}, $(X^{(N)},\mathcal{S}^{(N)}(T))$ is a PWI $\theta^{(N)}$-adapted to $(\lambda^{(N)},\pi^{(N)})$ and the restriction of $\gamma_{\theta}$ to $I^{(N)}$ is an isometric embedding of $(I^{(N)},f_{\lambda^{(N)},\pi^{(N)}})$ into $(X^{(N)},\mathcal{S}^{(N)}(T))$. Hence we have \eqref{ent4} which combined with \eqref{ent6} shows that $\theta^{(N)} = 0$. Therefore $\theta \in \bigcup_{n=0}^{+\infty} B_{\T^{\mathcal{A}}}^{(-n)}(\lambda,\pi)\cdot 0$,  which contradicts $\theta \in \mathcal{W}_{[\lambda],\pi}^{\delta}$. Thus $\gamma_{\theta}$ is not a linear embedding.
	
	This proves that $\gamma_{\theta}$ is a non-trivial isometric embedding of $(I,f_{\lambda,\pi})$ into $(X,T)$.
\end{proof}

\textbf{Acknowledgements.} The authors would like to thank Peter Ashwin, Michael Benedicks and Carlos Matheus for valuable suggestions and discussions.

\end{document}